\documentclass[a4paper,11pt,reqno]{amsart}

\usepackage{enumerate}
\usepackage[PS]{diagrams}
\usepackage{graphicx}

\usepackage{tikz}
\usetikzlibrary{decorations.markings}

\usepackage[hypertexnames=true,backref]{hyperref}

\newtheorem{theorem}{Theorem}[section]
\newtheorem{corollary}[theorem]{Corollary}
\newtheorem{proposition}[theorem]{Proposition}
\newtheorem{lemma}[theorem]{Lemma}

\newtheorem{maintheorem}[theorem]{Main theorem}

\theoremstyle{definition}
\newtheorem{definition}[theorem]{Definition}
\newtheorem{remark}[theorem]{Remark}
\newtheorem{construction}[theorem]{Construction}
\newtheorem{example}[theorem]{Example}

\newcommand{\Chd}[1]{\mathrm{Ch(#1\text{-Mod})}}
\newcommand{\bZ}{\mathbb{Z}}
\newcommand{\bR}{\mathbb{R}}
\newcommand{\bN}{\mathbb{N}}
\newcommand{\fp}{\mathcal{F}(S)}                   
\newcommand{\bsd}{{E}}           
\newcommand{\comp}[1]{A \langle {#1} \rangle} 
\newcommand{\nbd}{\nobreakdash}
\newcommand{\bfree}{B'}

\numberwithin{section}{part}   

\newcommand{\ou}[2]{\begin{diagram}[PS,h=1.2em]   {#1}\\{\comp{#2}}    \end{diagram}} 

\newcommand{\fs}{R[x,x\inv,y,y\inv]}
\newcommand{\fy}{R[x,x\inv]\powers{y\inv}}
\newcommand{\fx}{R[y,y\inv]\powers{x\inv}}
\newcommand{\fv}{R\powers{x\inv,y\inv}}
\newcommand{\fys}{R[x,x\inv]\nov{y\inv}}
\newcommand{\fxs}{R[y,y\inv]\nov{x\inv}}
\newcommand{\fvs}{R\powers{x\inv,y\inv}}
\newcommand{\fvy}{R\nov{x\inv}\powers{y\inv}}
\newcommand{\fvx}{R\nov{y\inv}\powers{x\inv}}
\newcommand{\fvys}{R\nov{x\inv}\nov{y\inv}}
\newcommand{\fvxs}{R\nov{y\inv}\nov{x\inv}}

\newcommand{\as}{xyR[x,y]}
\newcommand{\ay}{0}
\newcommand{\ax}{0}
\newcommand{\av}{0}
\newcommand{\ays}{xyR[x,y]}
\newcommand{\axs}{xyR[x,y]}
\newcommand{\avs}{xyR[x,y]}
\newcommand{\avy}{0}
\newcommand{\avx}{0}
\newcommand{\avys}{xyR[x,y]}
\newcommand{\avxs}{xyR[x,y]}

\newcommand{\bs}{yR[x\inv,y]}
\newcommand{\by}{0}
\newcommand{\bx}{yR[y]\powers{x\inv}}
\newcommand{\bv}{0}
\newcommand{\bys}{yR[x\inv,y]}
\newcommand{\bxs}{yR[y]\powers{x\inv}}
\newcommand{\bvs}{y{R\powers{x\inv}[y]}}
\newcommand{\bvy}{0}
\newcommand{\bvx}{yR[y]\powers{x\inv}}
\newcommand{\bvys}{yR\powers{x\inv}[y]}
\newcommand{\bvxs}{yR[y]\powers{x\inv}}

\newcommand{\cs}{R[x\inv,y\inv]}
\newcommand{\cy}{R[x\inv]\powers{y\inv}}
\newcommand{\cx}{R[y\inv]\powers{x\inv}}
\newcommand{\cv}{R\powers{x\inv,y\inv}}
\newcommand{\cys}{R[x\inv]\powers{y\inv}}
\newcommand{\cxs}{R[y\inv]\powers{x\inv}}
\newcommand{\cvs}{R\powers{x\inv,y\inv}}
\newcommand{\cvy}{R\powers{x\inv,y\inv}}
\newcommand{\cvx}{R\powers{x\inv,y\inv}}
\newcommand{\cvys}{R\powers{x\inv,y\inv}}
\newcommand{\cvxs}{R\powers{x\inv,y\inv}}

\newcommand{\ds}{xR[x,y\inv]}
\newcommand{\dx}{0}
\newcommand{\dy}{xR[x]\powers{y\inv}}
\newcommand{\dv}{0}
\newcommand{\dxs}{xR[x,y\inv]}
\newcommand{\dys}{xR[x]\powers{y\inv}}
\newcommand{\dvs}{xR\powers{y\inv}[x]}
\newcommand{\dvx}{0}
\newcommand{\dvy}{xR[x]\powers{y\inv}}
\newcommand{\dvxs}{xR\powers{y\inv}[x]}
\newcommand{\dvys}{xR[x]\powers{y\inv}}

\newcommand{\powers}[1]{[\kern -1pt [{#1}] \kern -1pt ]} 
\newcommand{\nov}[1]{(\kern -1.7pt ( {#1})\kern-1.7pt )} 

\let\nbd=\nobreakdash
\let\iso=\cong

\DeclareMathOperator*{\tensor}{\otimes}

\newcommand{\id}{\mathrm{id}}

\newcommand{\nix}{\,\text{-}\,}

\newcommand{\resp}{{\it resp.}}
\newcommand{\ie}{{\it i.e.}}
\newcommand{\totlt}{\ensuremath{{}{^\mathrm{lt}}\mathrm{Tot}\,}}
\newcommand{\totrt}{\ensuremath{\mathrm{Tot}^\mathrm{rt}\,}}
\newcommand{\totblt}{\ensuremath{{}_{\mathrm{blt}}\mathrm{Tot}\,}}
\newcommand{\totds}{\ensuremath{\mathrm{T}
    \raisebox {0.6pt} {\scalebox{0.6}{\ensuremath{+}}} \hskip 0.25 pt
    \llap{\rm o} \mathrm{t}\,}} 
\newcommand{\vgamma}{\check{\Gamma}}

\newcommand{\prodlt}{\ensuremath{\sideset{^\textrm{lt}}{}\prod}}
\newcommand{\prodrt}{\ensuremath{\sideset{}{^\textrm{rt}}\prod}}
\newcommand{\prodblt}{\ensuremath{\sideset{_\mathrm{blt}}{}\prod}}

\newcommand{\T}{\ensuremath{\mathcal{T}}}
\newcommand{\A}{\ensuremath{\mathcal{A}}}
\newcommand{\inv}{^{-1}}
\newcommand{\LL}{{R[x,\,x\inv,\, y,\, y\inv]}}

\begin{document}

\title[Finite domination and Novikov rings. Two variables]{Finite
  domination and Novikov rings.\\Laurent polynomial rings in two variables}

\date{\today}

\author{Thomas H\"uttemann}

\address{Thomas H\"uttemann\\ Queen's University Belfast\\ School of
  Mathematics and Physics\\ Pure Mathematics Research Centre\\ Belfast
  BT7~1NN\\ Northern Ireland, UK}

\email{t.huettemann@qub.ac.uk}

\urladdr{http://huettemann.zzl.org/}

\author{David Quinn}

\address{David Quinn\\ University of Aberdeen\\ School of Natural and
  Computing Sciences\\ Institute of Mathematics\\ Fraser Noble Building\\
  Aberdeen AB24~3UE\\ Scotland, UK}

\email{davidquinnmath@gmail.com}


\thanks{This work was supported by the Engineering and Physical
  Sciences Research Council [grant number EP/H018743/1].}

\begin{abstract}
  Let $C$ be a bounded cochain complex of finitely generated free
  modules over the \textsc{Laurent} polynomial ring $L = \LL$. The
  complex $C$ is called $R$-finitely dominated if it is homotopy
  equivalent over~$R$ to a bounded complex of finitely generated
  projective $R$-modules. Our main result characterises $R$-finitely
  dominated complexes in terms of \textsc{Novikov} cohomology: $C$ is
  $R$-finitely dominated if and only if eight complexes derived
  from~$C$ are acyclic; these complexes are $C \tensor_L
  R\powers{x,y}[(xy)\inv]$ and $C \tensor_L
  R[x,x\inv]\powers{y}[y\inv]$, and their variants obtained by
  swapping $x$ and~$y$, and replacing either indeterminate by its
  inverse.
\end{abstract}

\maketitle

\tableofcontents

\part{Introduction}

\section{The main theorem}

Let $R \subseteq K$ be a pair of unital rings. A cochain complex~$C$
of $K$\nbd-modules is called {\it $R$-finitely dominated\/} if $C$~is
homotopy equivalent, as an $R$\nbd-module complex, to a bounded
complex of finitely generated projective $R$\nbd-modules.

Finite domination is relevant, for example, in group theory and
topology. Suppose that $G$ is a group of type~$(FP)$; this means, by
definition, that the trivial $G$\nbd-module~$\bZ$ admits a finite
resolution~$C$ by finitely generated projective
$\bZ[G]$\nbd-modules. Let $H$~be a subgroup of~$G$. Deciding whether
$H$ is of type~$(FP)$ is equivalent to deciding whether $C$ is
$\bZ[H]$\nbd-finitely dominated.

In topology, finite domination is considered in the context of
homological finiteness properties of covering spaces, or properties of
ends of manifolds. See for example \textsc{Ranicki}'s article for
results and references~\cite{Ranicki-findom}.

\medbreak

Our starting point is the following result of \textsc{Ranicki}:

\begin{theorem}[{\cite[Theorem~2]{Ranicki-findom}}]
  Let $C$ be a bounded complex of finitely generated free modules over
  $K = R[x,x\inv]$. The complex $C$ is $R$\nbd-finitely dominated if
  and only if the two complexes
  \[C \tensor_K R\nov{x} \quad \text{and} \quad C \tensor_K
  R\nov{x\inv}\] are acyclic.
\end{theorem}

Here $R\nov{x} = R\powers{x}[x\inv]$ denotes the ring of formal
\textsc{Laurent} series in~$x$, and $R\nov{x\inv} =
R\powers{x\inv}[x]$ denotes the ring of formal \textsc{Laurent} series
in~$x\inv$.

For \textsc{Laurent} polynomial rings in several indeterminates, it is
possible to strengthen this result to allow for iterative application,
see for example~\cite{Iteration}. In particular, writing $L = \LL$ for
the \textsc{Laurent} polynomial ring in two variables, one can show
that a bounded complex of finitely generated free $L$\nbd-modules is
$R$\nbd-finitely dominated if and only if the four complexes
\begin{align*}
  &C \tensor_L R[x,x\inv]\nov{y}  && C \tensor_L R[x,x\inv]\nov{y\inv} \\
  &C \tensor_{R[x,x\inv]} R\nov{x} && C \tensor_{R[x,x\inv]} R\nov{x\inv}
\end{align*}
are acyclic.

In the present paper, we pursue a different, non-iterative approach,
leading to a new characterisation of finite domination. To state the
main result, we introduce another bit of notation: we write
$R\nov{x,y}$ for the ring of those formal \textsc{Laurent} series~$f$
in~$x$ and~$y$ having the property that $x^ky^k \cdot f \in
R\powers{x,y}$ for $k$ sufficiently large. That is,
\[R\nov{x,y} = R\powers{x,y} [(xy)\inv] \ .\]

\begin{maintheorem}
  \label{thm:main}
  Let $C$ be a bounded cochain complex of finitely generated free
  $L$\nbd-modules. Then the following two statements are equivalent:
  \begin{enumerate}[{\rm (a)}]
  \item \label{item:findom} The complex~$C$ is $R$\nbd-finitely
    dominated, \ie, $C$~is homotopy equivalent, as an
    $R$\nbd-module cochain complex, to a bounded cochain complex of
    finitely generated projective $R$\nbd-modules.
  \item \label{item:acyclic} The eight cochain complexes listed below
    are acyclic (all tensor products taken over~$L$):
    \begin{subequations}
      \begin{gather}\left.
          \begin{aligned}
            \label{eq:cond_edge}
            & C \tensor R[x,\, x\inv]\nov{y} && C \tensor R[x,\,
            x\inv]\nov{y\inv} \\
            & C \tensor R[y,\, y\inv]\nov{x} && C \tensor R[y,\,
            y\inv]\nov{x\inv}
          \end{aligned}\quad \right\}
        \\
        \noalign{\smallskip}\left.
          \begin{aligned}
            \label{eq:cond_vertex}
            & C \tensor R\nov{x,\, y} \hskip 2.1 em && C \tensor
            R\nov{x\inv,\, y\inv} \hskip 1.15 em \\
            & C \tensor R\nov{x,\, y\inv} && C \tensor R\nov{x\inv,\, y}
          \end{aligned}\quad \right\}
      \end{gather}
    \end{subequations}
  \end{enumerate}
\end{maintheorem}

The proofs of the two implications are quite different in nature. We
will establish \eqref{item:findom}~$\Rightarrow$~\eqref{item:acyclic}
in Corollaries~\ref{cor:x_or_y_acyclic} and~\ref{cor:XYacyclic} with
tools from homological algebra of multi-complexes, generalising
techniques used by the first author in~\cite{homology}, while the
reverse implication is treated in \S\ref{sec:final_proof} by a
homotopy theoretic argument, ultimately generalising one half of the
proof of \cite[Theorem~2]{Ranicki-findom}.

\section{Relation with $\Sigma$-invariants}

It might be worth explaining how our results are related to the
so-called $\Sigma$\nbd-invariants in the spirit of \textsc{Bieri},
\textsc{Neumann} and \textsc{Strebel}. Let $G$ be a group. For every
character $\chi \colon G \rTo \bR$ to the additive group of the reals
we have a monoid $G_\chi = \{g \in G \,|\, \chi(g) \geq 0\}$. Now
suppose $C$ is a non-negatively indexed chain complex of
$\bZ[G]$\nbd-modules. Then $C$~has, by restriction of scalars, the
structure of a $\bZ[G_\chi]$\nbd-module chain complex. Following
\textsc{Farber\/} {\it et.al.\/} one defines
\cite[Definition~9]{FGS-Sigma} the $m$th $\Sigma$\nbd-invariant of~$C$
as
\[\Sigma^m (C) = \{\chi \neq 0 \,|\, C \text{ has finite
  \(m\)-type over \(\bZ[G_\chi]\)} \} / \bR_+\ .\] So $\Sigma^m (C)$
is a quotient of the set of non-trivial~$\chi$ for which there is a
chain complex~$C^\prime$ consisting of finitely generated projective
$\bZ[G_\chi]$\nbd-modules, and a $\bZ[G_\chi]$-linear chain map $f
\colon C^\prime \rTo C$ with $f_i \colon H_i (C^\prime) \rTo H_i (C)$
an isomorphism for $i<m$ and an epimorphism for $i=m$. Two different
characters are identified in the quotient if, and only if, they are
positive real multiples of each other.

\begin{theorem}[{\cite[Corollary~4]{FGS-Sigma}}]
  \label{thm:fgs}
  Suppose that $C$ consists of finitely generated free
  $\bZ[G]$\nbd-modules, and is such that $C_i = 0$ for $i>m$. Let $N$
  be a normal subgroup of~$G$ with \textsc{abel}ian
  quotient~$G/N$. Then the $\bZ[N]$\nbd-module complex~$C$ is chain
  homotopy equivalent to a bounded chain complex of finitely generated
  projective $\bZ[N]$\nbd-modules concentrated in degrees $\leq m$ if
  and only if $\Sigma^m(C)$~contains the equivalence class of every
  non-trivial character of~$G$ that factorises through~$G/N$ (\ie,
  whose kernel contains~$N$).\qed
\end{theorem}

Note that the set of equivalence classes of non-trivial characters
which are trivial on~$N$ represents a sphere in the set of equivalence
classes of all non-trivial characters.

\smallbreak

An explicit link to \textsc{Novikov} homology has been documented by
\textsc{Sch\"utz}. For a character $\chi \colon G \rTo \bR$ we define
the \textsc{Novikov} ring
\[\widehat{R G_\chi} = \Big\{ f \colon G \rTo R \,|\, \forall t \in \bR
\colon \# \, \big( \mathrm{supp}(f) \cap \chi\inv ([t, \infty [\, )
\big) < \infty \Big\} \ ,\] equipped with the usual involution
product; here $\mathrm{supp}(f) = f\inv \big(R \setminus \{0\}
\big)$. Note that $R[G]$~is a subring of~$\widehat{R G_\chi}$.

\begin{theorem}[{\textsc{Sch\"utz} \cite[Theorem~4.7]{Schuetz-Sigma}}]
  \label{thm:schuetz}
  Let $C$ be a bounded chain complex of finitely generated free
  $R[G]$\nbd-modules. Suppose that $N$~is a normal subgroup of~$G$
  with quotient $G/N \iso \bZ^k$ a free \textsc{abel}ian group of
  finite rank. The complex $C$ is $R[N]$\nbd-finitely dominated if and
  only if for every character $\chi \colon G \rTo \bR$ which is
  trivial on~$N$ the complex $C \tensor_{R[G]} \widehat{R G_\chi}$ is
  acyclic.\qed
\end{theorem}

In these terms, our main result Theorem~\ref{thm:main} concerns the
case of a split group extension
\[1 \rTo N \rTo N \times \bZ^2 \rTo \bZ^2 \rTo 0\] (with the
translation $R = \bZ[N]$ and $L = \bZ[N \times \bZ^2] \iso \LL$), or
indeed the case $G = \bZ^2$ and trivial group~$N$.  The sphere which
characterises finite domination of~$C$ according to
Theorem~\ref{thm:fgs} is homeomorphic to a circle~$S^1$. In our
approach, it is replaced by the boundary of a square in~$\bR^2$; its
geometry and combinatorics encode algebraic information that lead to
our new set of homological conditions characterising finite
domination.

\section{Algebraic examples}

Let us work over the ring $R = \bZ$ for simplicity. We omit the
verification of the following purely algebraic facts:

\begin{lemma}
  \label{lem:units_in_Novikov}
  \begin{enumerate}[{\rm (a)}]
  \item The ring $\bZ\nov{x,y}$ is a unique factorisation domain. An
    element of~$\bZ\nov{x,y}$ is a unit if and only if it is a product
    of a monomial $x^k y^\ell$ (for some $k, \ell \in \bZ$) with a
    unit in~$\bZ\powers{x,y}$ (\ie, a power series having the constant
    term~$\pm 1$).
  \item The ring $\bZ[x,x\inv]\nov{y}$ is a unique factorisation
    domain. An element of~$\bZ[x,x\inv]\nov{y}$ is a unit if and only
    if it is a product of a monomial $y^\ell$ (for some $\ell \in
    \bZ$) with a unit in~$\bZ[x,x\inv]\powers{y}$ (\ie, a power series
    with $y^0$~having coefficient~$\pm x^k$ for some $k \in
    \bZ$). \qed
  \end{enumerate}
\end{lemma}

\begin{example}
  Let $\mu \in \bZ[x,x\inv,y,y\inv]$. The two-step cochain complex~$C$
  given by
  \[\ldots \rTo 0 \rTo \bZ[x,x\inv,y,y\inv] \rTo^{\cdot \mu}
  \bZ[x,x\inv,y,y\inv] \rTo 0 \rTo \ldots\] is $\bZ$\nbd-finitely
  dominated if and only if $\mu$~is a \textsc{Laurent} monomial with
  coefficient~$\pm 1$ (that is, is a unit).
\end{example}

\begin{proof}
  The case $\mu= 0$ is trivial. Otherwise, $C$~is $\bZ$\nbd-finitely
  dominated if and only if the chain complexes listed
  in~\eqref{eq:cond_edge} and~\eqref{eq:cond_vertex} are acyclic,
  which happens if and only if multiplication by~$\mu$ is surjective
  after tensoring with the eight rings $R[x,x\inv]\nov{y^{\pm 1}}$,
  $R[y,y\inv]\nov{x^{\pm 1}}$ and $R\nov{x^{\pm 1}, y^{\pm 1}}$; note
  that injectivity is automatic. This happens if and only if
  $\mu$~becomes a unit in all the rings in questions, which happens if
  and only if $\mu$~is a \textsc{Laurent} monomial with coefficient
  $\pm 1$, by Lemma~\ref{lem:units_in_Novikov}.
\end{proof}

\begin{example}
  \label{x:square_domination}
  Let $\mu = 1 + x \big( y^2 + x(1+y) \big)$ and $\nu =
  x+y\big(1+y(1+x^2) \big)$. The non-acyclic cochain complex~$C$ with
  non-trivial part
  \[\bZ[x,x\inv,y,y\inv] \rTo^\alpha \begin{matrix}
    \bZ[x,x\inv,y,y\inv] \\ \oplus \\
    \bZ[x,x\inv,y,y\inv] \end{matrix} \rTo^\beta
  \bZ[x,x\inv,y,y\inv]\] given by
  \[\alpha = \begin{pmatrix}\mu\\\nu\end{pmatrix} \quad \text{and}
  \quad \beta = \big( -\nu,\, \mu \big)\] is $\bZ$-finitely
  dominated.
\end{example}

\begin{proof}
  Up to isomorphism and re-indexing, the cochain complex can be
  obtained by computing the iterated mapping cone of the square
  diagram
  \begin{displaymath}
    \begin{diagram}[small,labelstyle=\scriptstyle]
      \bZ[x,x\inv,y,y\inv] & \rTo[l>=2em]^\mu & \bZ[x,x\inv,y,y\inv] \\
      \dTo<\nu && \dTo<\nu \\
      \bZ[x,x\inv,y,y\inv] & \rTo^\mu & \bZ[x,x\inv,y,y\inv]
    \end{diagram}
    \ ,
  \end{displaymath}
  \ie, by taking algebraic mapping cones horizontally (\resp,
  vertically) first, and then taking the mapping cone of the resulting
  vertical (\resp, horizontal) map. Using
  Lemma~\ref{lem:units_in_Novikov} we see that the map $\mu$ becomes
  an isomorphism after tensoring (over~$\bZ[x,x\inv,y,y\inv]$) with
  the rings $\bZ\nov{x,y}$, $\bZ\nov{x\inv,y}$,
  $\bZ[x,x\inv]\nov{y\inv}$, and $\bZ[y,y\inv]\nov{x}$, while $\nu$
  becomes an isomorphism after tensoring with any of the rings
  $\bZ\nov{x,y\inv}$, $\bZ\nov{x\inv,y\inv}$, $\bZ[x,x\inv]\nov{y}$,
  and $\bZ[y,y\inv]\nov{x\inv}$. Consequently in all cases the
  iterated mapping cone, which is isomorphic to $C$ tensored with the
  ring under consideration, will be an acyclic complex. The claim now
  follows from the Main Theorem.
\end{proof}

\section{Finitely dominated covering spaces}

\begin{proposition}
  \label{prop:fd_covering}
  Let $X$ be a connected finite $CW$ complex with universal covering
  space~$\tilde X$ and fundamental group $\pi_1 X = G \times \bZ^2$
  for some group~$G$. Let $Y \rTo X$ be the covering determined by the
  projection $\pi_1 X \rTo \bZ^2$, and let $C$~denote the cochain
  complex which is, up to the re-indexing $C^k = C_{-k}$, the cellular
  $\bZ[\pi_1 X]$\nbd-free chain complex of~$\tilde X$. Then $Y$~is a
  finitely dominated space (\ie, is a retract up to homotopy of a
  finite $CW$\nbd-complex) if and only if all the complexes listed
  in~\eqref{eq:cond_edge} and~\eqref{eq:cond_vertex} are acyclic.
\end{proposition}

\begin{proof}
  A connected $CW$ complex~$Z$ is finitely dominated if and only if
  $\pi_1 Z$ is finitely presented, and the cellular chain complex of
  the universal covering space of~$Z$ is homotopy equivalent to a
  bounded complex of finitely generated projective $\bZ[\pi_1
  Z]$\nbd-modules, cf.~\cite[\S3]{Ranicki-findom}.

  We apply this criterion to the space $Z = Y$, noting that $\pi_1 Y =
  G$, and that $\tilde X$~is the universal covering space
  of~$Y$. Since $X$~is a finite $CW$ complex, its fundamental group $G
  \times \bZ^2$ is finitely presented; consequently, its retract~$G$
  is finitely presented as well \cite[Lemma~1.3]{Wall-finiteness}.

  We identify the rings $\bZ[G \times \bZ^2]$
  and~$\bZ[G][x,x\inv,y,y\inv]$ by saying that first and second unit
  vector in~$\bZ^2$ correspond to the indeterminates~$x$ and~$y$,
  respectively. Since $G$ is finitely presented, it is enough to show
  that the cochain complex~$C$ is $\bZ[G]$\nbd-finitely dominated,
  which can be detected cohomologically by the Main Theorem applied to
  the complex~$C$ and the ring $R = \bZ[G]$.
\end{proof}

\begin{example}
  Let $Y_0$ be the real plane $\bR^2$, considered as a $CW$ complex
  with $0$\nbd-cells the integral points, $1$\nbd-cells joining
  vertically or horizontally adjacent $0$\nbd-cells, and $2$\nbd-cells
  the integral unit squares. The group~$\bZ^2$ acts evenly\footnote{An
    {\it even\/} action is, by definition, a free and properly
    discontinuous action.} by translation. By suitably attaching
  $\bZ^2$-indexed collections of cells of dimensions $4$, $5$ and~$6$
  we obtain a simply-connected space~$Y$ with an even action of~$\bZ^2$
  such that its cellular cochain complex is
  \begin{itemize}
  \item the cellular complex of~$\bR^2$ in chain levels $0$, $1$
    and~$2$;
  \item the chain complex of Example~\ref{x:square_domination} in
    chain levels $4$, $5$ and~$6$;
  \item trivial otherwise.
  \end{itemize}
  Let $X = Y/\bZ^2$; the projection map $Y \rTo X$ is the universal
  cover, and $\pi_1X = \bZ^2$. It follows from
  Example~\ref{x:square_domination} and
  Proposition~\ref{prop:fd_covering} that $Y$~is a finitely dominated
  $CW$ complex.
\end{example}

\part{Conventions. Multi-complexes}

Throughout the paper we let $R$ denote a unital associative ring
(possibly non-commutative). Modules will be tacitly understood to be
unital right modules, unless a different convention is specified. We
will use cochain complexes of $R$\nbd-modules, or of modules over a
ring related to~$R$; that is, our differentials {\it increase\/} the
degree.

\section{Lower triangular cochain complexes}

\begin{definition}
  \label{def:triangle_complexes}
  Let $n \geq 1$. A {\it lower $n$\nbd-triangular cochain complex\/}
  consists of a cochain complex~$C$ and for each $q \in \bZ$ a module
  isomorphism $\phi_q \colon C^q \iso \bigoplus_{p=1}^n C^{p,q}$ such
  that for each $q$ the composition~$d_{\ell,k}$
  \[C^{k,q} \rTo \bigoplus_{p=1}^n C^{p,q} \rTo[l>=3em]^{\phi_q\inv}
  C^q \rTo C^{q+1} \rTo[l>=3em]^{\phi_{q+1}} \bigoplus_{p=1}^n
  C^{p,q+1} \rTo C^{\ell, q+1}\] is the zero map whenever $n \geq k >
  \ell \geq 1$.
\end{definition}

In other words, under the isomorphisms $\phi_q$ the differential
of~$C$ becomes a lower triangular matrix of the form
\[D = 
\begin{pmatrix}
  d_{1,1} & 0 & 0 & 0 & \cdots & 0 \\
  d_{2,1} & d_{2,2} & 0 & 0 & \cdots & 0 \\
  d_{3,1} & d_{3,2} & d_{3,3} & 0 & \cdots & 0 \\
  \vdots & \vdots & \vdots & \ddots & \ddots & 0 \\
  d_{n-1,1} & d_{n-1,2} & \cdots & \cdots & d_{n-1,n-1} & 0 \\
  d_{n,1} & d_{n,2} & \cdots & \cdots & d_{n,n-1} & d_{n,n}
\end{pmatrix}\] with $d_{\ell,k} \colon C^{k,q} \rTo C^{\ell, q+1}$
satisfying $D \circ D = 0$. We have $d_{k,k} \circ d_{k,k} = 0$ for $1
\leq k \leq n$ so that {\it $C^{k,*}$ is a cochain complex with
differential~$d_{k,k}$} .

\begin{definition}
  \label{def:triangle_maps}
  Let $C$ and~$D$ be lower $n$\nbd-triangular cochain complexes, with
  structure isomorphisms $\phi_q$ and~$\psi_q$, respectively. A {\it
    map $f \colon C \rTo D$ of lower triangular cochain complexes\/}
  is a map of cochain complexes~$f$ such that the
  composition~$f_{\ell, k}$
  \[C^{k,q} \rTo \bigoplus_{p=1}^n C^{p,q} \rTo^{\phi_q\inv} C^q \rTo^f
  D^q \rTo^{\psi_q} \bigoplus_{p=1}^n D^{p,q} \rTo D^{\ell, q}\]
  is the zero map whenever $n \geq k > \ell \geq 1$.
\end{definition}

In other words, under the isomorphisms $\phi_q$ and~$\psi_q$ the map~$f$
becomes a lower triangular matrix of the form
\[F = 
\begin{pmatrix}
  f_{1,1} & 0 & 0 & 0 & \cdots & 0 \\
  f_{2,1} & f_{2,2} & 0 & 0 & \cdots & 0 \\
  f_{3,1} & f_{3,2} & f_{3,3} & 0 & \cdots & 0 \\
  \vdots & \vdots & \vdots & \ddots & \ddots & 0 \\
  f_{n-1,1} & f_{n-1,2} & \cdots & \cdots & f_{n-1,n-1} & 0 \\
  f_{n,1} & f_{n,2} & \cdots & \cdots & f_{n,n-1} & f_{n,n}
\end{pmatrix}\] with $f_{\ell, k} \colon C^{k,q} \rTo
D^{\ell,q}$. Then {\it $f_{k,k}$ is a cochain map from~$C^{k,*}$
  to~$D^{k,*}$}.

\begin{lemma}
  \label{lem:triangle_maps}
  Let $f \colon C \rTo D$ be a map of lower $n$\nbd-triangular cochain
  complexes. Suppose that, in the notation used above, the maps
  \[f_{k,k} \colon C^{k,*} \rTo D^{k,*} \qquad \text{(\(1 \leq k \leq
    n\))}\] are quasi-isomorphisms of cochain complexes. Then $f$~is a
  quasi-iso\-mor\-phism.
\end{lemma}

\begin{proof}
  The Lemma is a tautology for $n=1$. --- For $1 \leq k \leq n$ define
  a cochain complex $C(k)$ by setting
  \[C(k)^q = \bigoplus_{p=k}^n C^{p,q} \ ,\]
  equipped with differential 
  \[\begin{pmatrix}
    d_{k,k} & 0 & 0 & 0 & \cdots & 0 \\
    d_{k+1,k} & d_{k+1,k+1} & 0 & 0 & \cdots & 0 \\
    d_{k+2,k} & d_{k+2,k+1} & d_{k+2,k+2} & 0 & \cdots & 0 \\
    \vdots & \vdots & \vdots & \ddots & \ddots & 0 \\
    d_{n-1,k} & d_{n-1,k+1} & \cdots & \cdots & d_{n-1,n-1} & 0 \\
    d_{n,k} & d_{n,k+1} & \cdots & \cdots & d_{n,n-1} & d_{n,n}
  \end{pmatrix}\] (this is the lower right-hand part of the matrix~$D$
  above). We note that $C(k+1)$~is a subcomplex of~$C(k)$, and that
  the quotient $C(k)/C(k+1)$ is nothing but~$C^{k,*}$. Clearly $C(n) =
  C^{n,*}$, and $C(1)$~is isomorphic to~$C$ via the structure
  isomorphisms~$\phi_q$.

  We define the analogous objects $D(k)$ associated to the cochain
  complex~$D$; the remarks on the~$C(k)$ apply {\it mutatis mutandis}.

  The map~$f$ (or rather, its matrix representation~$F$) restricts to
  cochain complex maps $f(k) \colon C(k) \rTo D(k)$, with $f(n) =
  f_{n,n}$ and $f(1)$~being isomorphic to~$f$. The maps $f(k)$ fit
  into commutative ladder diagrams
  \begin{diagram}[PS]
    0 & \rTo & C(k+1) & \rTo & C(k) & \rTo & C^{k,*} & \rTo & 0 \\
    && \dTo<{f(k+1)} && \dTo<{f(k)} && \dTo<{f_{k,k}} \\
    0 & \rTo & D(k+1) & \rTo & D(k) & \rTo & D^{k,*} & \rTo & 0 \\      
  \end{diagram}
  for $1 \leq k < n$, inducing a ladder diagram of long exact
  cohomology sequences as usual. Using the Five Lemma, and the fact
  that the $f_{k,k}$ are known to be quasi-isomorphisms by hypothesis,
  we conclude in turn that the maps $f(n-1),\ f(n-2),\ \cdots,\
  f(1)\iso f$ are quasi-isomorphisms, thereby proving the Lemma.
\end{proof}

\section{Double complexes}
\label{sec:double-complexes}

A {\it double complex\/} $D^{*,*}$ is a $\bZ \times \bZ$\nbd-indexed
collection $\big(D^{p,q}\big)_{p,q \in \bZ}$ of right $R$\nbd-modules
together with ``horizontal'' and ``vertical'' differentials
\[d_h \colon D^{p,q} \rTo D^{p+1,q} \quad \text{and}
\quad d_v \colon D^{p,q} \rTo D^{p,q+1}\] which satisfy the conditions
\[d_h \circ d_h = 0 \ , \quad d_v \circ d_v = 0 \ , \quad d_h \circ
d_v + d_v \circ d_h =0\ .\] Note that the differentials
anti-commute. We will in general consider unbounded double complexes
so that $D^{p,q} \neq 0$ may occur for $|p|$ and $|q|$ arbitrarily
large.

\begin{definition}\label{def:directsumtot}
  Let $D^{\ast,\ast}$ be a double complex. We define its {\em direct
    sum totalisation} to be the cochain complex $\totds D^{\ast,\ast}$
  which in cochain level $n$ is given by the direct sum \[ \big(\totds
  D^{\ast,\ast}\big)^n = \bigoplus_p D^{p,n-p}\,;\] the differential
  is given by $d_h+d_v$, where $d_h$ and $d_v$ are the ``horizontal''
  and ``vertical'' differentials of $D^{\ast,\ast}$ respectively.
\end{definition}

We will make use of the following standard result when comparing
double complexes and the direct sum totalisation of each.

\begin{lemma}
  \label{lemma:dblcpx}
  Let $h\colon D^{\ast,\ast}\rTo E^{\ast,\ast}$ be a map of double
  complexes which are concentrated in finitely many columns. If $h$ is
  a quasi-isomorphism on each column or on each row, then the induced
  map \[\totds(h)\colon \totds D^{\ast,\ast} \rTo \totds
  E^{\ast,\ast}\] is a quasi-isomorphism.
\end{lemma}

\begin{proof}
  Let us first deal with the case that $h$ is a quasi-isomorphism on
  each column. Note that $\totds D^{\ast,\ast}$ and $\totds
  E^{\ast,\ast}$ can be given the structure of lower
  $n$\nbd-triangular complexes in the sense of
  Definition~\ref{def:triangle_complexes}, for the same~$n$, such that
  $\totds (h)$~is a map of lower triangular complexes in the sense of
  Definition~\ref{def:triangle_maps}. In more detail, let us assume
  that $D^{\ast,\ast}$ and~$E^{\ast,\ast}$ are concentrated in columns
  $1$ to~$n$, for ease of indexing. Then $\big( \totds D^{\ast,\ast}
  \big)^q = \bigoplus_{p=1}^n D^{p, q-p}$ so that $C = \totds
  D^{\ast,\ast}$ has the lower triangular decomposition
  \[C^q = \bigoplus_{p=1}^n C^{p,q} \quad \text{where} \quad C^{p,q} =
  D^{p,p-q} \ .\] A similar decomposition can be defined for~$\totds
  E^{\ast,\ast}$. The hypothesis that $h$~is a quasi-isomorphism on
  each column translates into the hypothesis of
  Lemma~\ref{lem:triangle_maps} which thus implies that $\totds
  (h)$~is a quasi-isomorphism as claimed.

  \smallskip

  If $h$ is a quasi-isomorphism on each row we can apply the same
  reasoning with the roles of rows and columns reversed, provided our
  original complex is bounded in the vertical direction as well. In
  the general case, one can appeal to a spectral sequence
  argument. More precisely, ``filtration by rows'' gives rise to a
  convergent spectral sequence
  \[E_1^{*,*} = H^v (D^{*,*}) \Longrightarrow H^* (\totds D^{*,*}) \
  ,\] and similarly for the bicomplex~$E^{*,*}$; the map $f$ induces a
  map of spectral sequences which is an isomorphism on $E^1$-terms,
  and thus induces a quasi-isomorphism on abutments as claimed.
\end{proof}

\begin{lemma}
  \label{lem:augmented_double}
  Suppose that the double complex $E^{\ast,\ast}$ is concentrated in
  the first quadrant (that is, suppose that $E^{p,q}=0$ if $p<0$ or
  $q<0$). Suppose further that we are given a cochain complex $C$ with
  $C^q = 0$ if $q<0$, and maps $h_q \colon C^q \rTo E^{0,q}$ such that
  for each $q \geq 0$ the sequence
  \[0 \rTo C^q \rTo^{h_q} E^{0,q} \rTo E^{1,q} \rTo E^{2,q} \rTo \cdots\]
  is exact (\ie, is an acyclic cochain complex), and such that the
  composites
  \[C^q \rTo E^{0,q} \rTo E^{0,q+1} \quad\text{and}\quad C^q \rTo
  C^{q+1} \rTo E^{0,q+1}\] agree up to sign. Then there is a
  quasi-isomorphism
  \[C \rTo \totds E^{\ast,\ast}\] induced by the~$h_q$.  
\end{lemma}

\begin{proof}
  The proof given by \textsc{Bott} and~\textsc{Tu}
  \cite[p.~97]{Bott-Tu} applies to the current situation; see the
  remark following Proposition~8.8 of {\it loc.cit.}.
\end{proof}

\section{Triple complexes}
\label{sec:triple-complexes}

\begin{definition}
  A triple complex $T^{*,*,*}$ is a $\bZ^3$\nbd-indexed family of
  modules $T^{x,y,z}$ together with three anti-commuting differentials
  \[d_x \colon T^{x,y,z} \rTo T^{x+1,y,z} \ , \ d_y \colon T^{x,y,z}
  \rTo T^{x,y+1,z} \ , \ d_z \colon T^{x,y,z} \rTo T^{x,y,z+1} \ .\]
  ``Anti-commuting'' means that the differentials satisfy $d_id_j =
  (\delta_{ij} -1) d_jd_i$.
\end{definition}

The {\em direct sum totalisation} $\totds T^{\ast,\ast,\ast}$ is the
cochain complex with
\[\big(\totds T^{\ast,\ast,\ast}\big)^n = \bigoplus_{x+y+z=n}
T^{x,y,z}\] and differential $d=d_x+d_y+d_z$. We use the same notation
as for the direct sum totalisation of double complexes.

\begin{definition}
  \label{def:partial_tot}
  For a triple complex $T^{\ast,\ast,\ast}$ we denote by
  $\mathrm{Tot}_{x,y}\,T^{\ast,\ast,\ast}$ the {\em partial
    totalisation} of $T^{\ast,\ast,\ast}$ with respect to $x$ and $y$.
  We define this to be the double complex given by
  \[ \big(\mathrm{Tot}_{x,y}\,T^{\ast,\ast,\ast}\big)^{p,q} =
  \bigoplus_{x+y=p} T^{x,y,q} \] with ``horizontal'' differential $d_h
  = d_x+d_y$ and ``vertical'' differential $d_v = d_z$.
\end{definition}

It is easy to verify that
\[\totds\big(\mathrm{Tot}_{x,y}\,T^{\ast,\ast,\ast}\big) = \totds
T^{\ast,\ast,\ast} \ ,\] where $\totds$ denotes the direct sum
totalisation of either a double or a triple complex as determined by
context.

\begin{lemma}
  \label{lem:triple_map}
  Let $f \colon T^{\ast,\ast,\ast} \rTo U^{\ast,\ast,\ast}$ be a map
  of triple complexes concentrated in a finite cubical region
  of~$\bZ^3$. Suppose that $f$~is a quasi-isomorphism of all cochain
  complexes in $z$\nbd-direction (\resp, in $y$\nbd-direction, \resp,
  in $x$\nbd-direction), that is, suppose that
  \[f \colon T^{x,y,*} \rTo U^{x,y,*} \text{ (\resp, } T^{x,*,z} \rTo
  U^{x,*,z}, \text{ \resp, } T^{*,y,z} \rTo U^{*,y,z} \text{ )}\] is a
  quasi-isomorphism for all $x,y \in \bZ$ (\resp, all $x,z \in \bZ$,
  \resp, all $y,z \in \bZ$). Then
  \[\totds (f) \colon \totds T^{\ast,\ast,\ast} \rTo \totds
  U^{\ast,\ast,\ast}\] is a quasi-isomorphism.
\end{lemma}

\begin{proof}
  If $f \colon T^{x,y,*} \rTo U^{x,y,*}$ is a quasi-isomorphism for
  all $x,y \in \bZ$, then $\mathrm{Tot}_{x,y}(f) \colon
  \mathrm{Tot}_{x,y} T^{\ast,\ast,\ast} \rTo \mathrm{Tot}_{x,y}
  U^{\ast,\ast,\ast}$ is a map of double complexes satisfying the
  hypotheses of Lemma~\ref{lemma:dblcpx}. Consequently,
  \[\totds (f) = \totds \mathrm{Tot}_{x,y} (f) \colon \totds
  T^{\ast,\ast,\ast} \rTo \totds U^{\ast,\ast,\ast}\] is a
  quasi-isomorphism. --- The other cases can be proved in a similar
  manner.
\end{proof}

\part{Finite domination implies acyclicity}
\label{part:1}

\section{Truncated products. One variable}
\label{sec:realisation}

Let $R$ be a ring with unit, and let $z$ be an indeterminate. We have
obvious $R$\nbd-module isomorphisms
\[R[z] \iso \bigoplus_{\bN} R \ , \quad R[z,z\inv] \iso
\bigoplus_{\bZ} R \ , \quad R\powers{z} \iso \prod_\bN R\] mapping $rz^k$
to the element~$r$ of the $k$th summand or factor on the right. Using
these isomorphisms, we may write elements of these infinite sums or
products as polynomials, \textsc{Laurent} polynomials and formal power
series in~$z$, respectively. Similarly, a formal \textsc{Laurent}
power series in~$z$ (involving finitely many negative powers of~$z$)
corresponds to an element of a ``truncated product'' via the obvious
$R$\nbd-module isomorphism
\[R\powers{z}[z\inv] \iso \bigoplus_{i < 0} R \oplus \prod_{i
  \geq 0} R \ .\] This ring is known as a \textsc{Novikov}
ring, as is its counterpart with $z\inv$ in place of~$z$; we reserve
the notation
\[R\nov{z} = R\powers{z}[z\inv] \quad \text{and} \quad R\nov{z\inv} =
R\powers{z\inv}[z] \ .\]

There is a module theoretic version corresponding to the construction
of the \textsc{Novikov} ring~$R\nov{z}$. Given a $\bZ$\nbd-indexed
family of modules $M_i$ we define the {\it left truncated product\/}
to be the module
\[\prodlt_i M_i = \bigoplus_{i<0} M_i \,\oplus\, \prod_{i \geq 0} M_i \ ;\]
the elements of this truncated product will be written as formal
\textsc{Laurent\/} series $\sum_{i \geq k} m_i z^i$ with $m_i \in
M_i$.  We let $M\nov{z}$ denote the module of formal \textsc{Laurent}
series with coefficients in~$M$,
\[M\nov{z} = \prodlt M = \Big\{ \sum_{i \geq k} m_i z^i \,|\, k \in
\bZ,\ m_i \in M \Big\} \ .\] Note that $M\nov{z}$ carries a canonical
$R\nov{z}$\nbd-module structure described by $z \cdot \sum_{i \geq k}
m_i z^i = \sum_{i \geq k} m_i z^{i+1}$. --- Dually we define the {\it
  right truncated product\/} to be the module
\[\prodrt_i M_i = \prod_{i \leq 0} M_i \,\oplus\, \bigoplus_{i>0}
M_i\] of formal \textsc{Laurent} series which are finite to the
right, and define $M\nov{z\inv}$ by setting
\[M\nov{z\inv} = \prodrt M = \Big\{ \sum_{i \leq k} m_i z^i \,|\, k \in
\bZ,\ m_i \in M \Big\} \ .\] The module $M\nov{z\inv}$ carries an
obvious $R\nov{z\inv}$\nbd-module structure described by $z\inv \cdot
\sum_{i \leq k} m_i z^i = \sum_{i \leq k} m_i z^{i-1}$.

\begin{lemma}
  \label{lem:series}
  Suppose that $M$ is a finitely presented right $R$\nbd-module. There
  is a natural isomorphism of $R\nov{z}$\nbd-modules
  \[\Phi_M \colon M \tensor_R R\nov{z} \rTo^\iso M\nov{z} \ , \quad m
  \tensor \sum_{i \geq k} r_i z^i \mapsto \sum_{i \geq k} mr_i z^i \
  , \] and a similar isomorphism $\Psi_M \colon M \tensor_R
  R\nov{z\inv} \rTo^\iso M\nov{z\inv}$.
\end{lemma}

\begin{proof}
  The proof is standard, details can be found, for example,
  in~\cite[Lemma~2.1]{homology}. One establishes the result for
  finitely generated free $R$\nbd-modules first, and then passes to
  the general case by considering a two-step resolution of~$M$ by
  finitely generated free modules.
\end{proof}

\section{Truncated product totalisation of double complexes}
\label{sec:trtot}

We will consider non-standard totalisation functors for double
complexes formed by using truncated products; this technology has been
used in~\cite{homology} to analyse \textsc{Novikov} cohomology of
algebraic mapping tori. We begin by recalling some definitions and
results useful for our present purposes; later we will extend these
ideas to \textsc{Laurent} rings with two variables.

\begin{definition}
  Let $D^{*,*}$ be a double complex. We define its {\it left truncated
    totalisation\/} to be the cochain complex $\totlt D^{*,*}$ which
  in cochain level~$n$ is given by the left truncated product
  \[\big( \totlt D^{*,*} \big)^n = \prodlt_p D^{p, n-p} \, ;\] 
  the differential~$d$ is induced in an obvious way by the horizontal
  differential $d_h \colon D^{p,q} \rTo D^{p+1,q}$ and the vertical
  differential $d_v \colon D^{p,q} \rTo D^{p,q+1}$: the component
  mapping into the $p$th factor of $\big(\totlt D^{*,*}\big)^{n+1}$ is
  the sum of the horizontal differential coming from the $(p-1)$st
  factor of $\big( \totlt D^{*,*} \big)^n$, and the vertical
  differential coming from the $p$th factor of $\big( \totlt D^{*,*}
  \big)^n$. Explicitly, for an element $x = \sum_{p \geq k} a_p z^p$
  of $\big( \totlt D^{*,*} \big)^n$ we have
  \[d(x) = \sum_{p \geq k} \big( d_h (a_{p-1}) + d_v(a_p) \big) z^p\]
  (where we set $a_{k-1} = 0$ for convenience).

  Dually, we define the {\it right truncated totalisation\/} to be the
  cochain complex $\totrt D^{*,*}$ which in cochain level~$n$ is given
  by the right truncated product
  \[\big( \totrt D^{*,*} \big)^n = \prodrt_p D^{p,n-p}\] with
  differential described by
  \[d \colon \sum_{p \leq k} a_p z^p \mapsto \sum_{p \leq k+1} \big(
  d_h (a_{p-1}) + d_v(a_p) \big) z^p \ .\]
\end{definition}

\begin{proposition}[{\cite[Corollary~6.7]{Bergman}},
  {\cite[Proposition~1.2]{homology}}]
  \label{prop:double_acyclic}
  Suppose the double complex $D^{*,*}$ has exact columns. Then $\totlt
  D^{*,*}$ is acyclic. Dually, if $D^{*,*}$ has exact rows then $\totrt
  D^{*,*}$ is acyclic.
\end{proposition}

\begin{proof}
  This can be proved by an elementary diagram chase, associating to
  each cocycle~$m$ in $\big(\totlt D^{*,*}\big)^n$ (\resp, in
  $\big(\totrt D^{*,*}\big)^n$) an element in the module $\big(\totlt
  D^{*,*}\big)^{n-1}$ (\resp, in $\big(\totrt D^{*,*}\big)^{n-1}$)
  with coboundary~$m$. Details can be found in the given references.
\end{proof}

\section{Algebraic mapping $1$-tori}
\label{sec:1-tori}

\begin{definition}
  Let $C$ be a cochain complex of right $R$\nbd-modules, and let $h
  \colon C \rTo C$ be a cochain map. The {\it mapping $1$\nbd-torus
    $\T(h)$ of~$h$} is the $R[z,z\inv]$\nbd-module cochain complex
  \[\T(h) = \mathrm{Cone}\, \big( C \tensor_R R[z,z\inv]
  \rTo[l>=5em]^{h \tensor \id - \id \tensor z} C \tensor_R R[z,z\inv]
  \big)\] where the map ``$z$'' denotes the self map of $R[z,z\inv]$
  given by multiplication by the indeterminate~$z$.
\end{definition}

In this definition ``Cone'' stands for the algebraic mapping cone; if
a map of cochain complexes $f \colon X \rTo Y$ is considered as a
double complex~$D^{*,*}$ concentrated in columns $p=-1,0$ with
horizontal differential~$f$ and the differential of~$X$ modified by
the factor $-1$, then $\mathrm{Cone}\,(f) = \totds
D^{*,*}$. Explicitly, we have $\mathrm{Cone}\, (f)^n = X^{n+1} \oplus
Y^n$, and the differential is given by the following formula:
\begin{align*}
  \mathrm{Cone}\,(f)^n = X^{n+1} \oplus Y^n &\rTo X^{n+2} \oplus
  Y^{n+1} = \mathrm{Cone}\,(f)^{n+1} \\
  (x,\,y) & \mapsto \quad \big(-d (x),\, f(x) + d (y)\big)
\end{align*}

\begin{proposition}
  \label{prop:prop_mt1}
  \begin{enumerate}[{\rm (1)}]
  \item \label{item:T_hom_iso} For homotopic maps $f, g \colon C \rTo
    C$, the mapping $1$-tori $\T(f)$ and $\T(g)$ are isomorphic.
  \item \label{item:T_mather1} \textsc{Mather} trick: For maps $f
    \colon C \rTo D$ and $g \colon D \rTo C$ of cochain complexes, the
    mapping $1$-tori $\T(fg)$ and $\T(gf)$ are homotopy equivalent.
  \item If $C$ is a bounded above complex of $R[z,z\inv]$\nbd-modules
    which are projective as $R$\nbd-modules, then $C$ and $\T(z)$ are
    homotopy equivalent as $R[z,z\inv]$\nbd-module complexes. Here
    ``$z$'' denotes the $R$\nbd-linear self map given by
    multiplication with the indeterminate~$z$, and the mapping
    $1$-torus is formed by considering~$C$ as an $R$\nbd-module
    complex.
  \end{enumerate}
\end{proposition}

\begin{proof}
  This can be found (using the language of chain complexes rather than
  cochain complexes), for example, in \cite[\S2]{Iteration}.
\end{proof}

\section{Mapping 1-tori and totalisation. Applications}

Let $h \colon C \rTo C$ be a self map of a cochain complex~$C$.
Observe that by additivity of tensor products we have an equality of
cochain complexes
\begin{equation}
  \label{eq:mt1}
  \T(h) \tensor_{R[z,z\inv]} R\nov{z} = \mathrm{Cone}\, \big( C \tensor_R
  R\nov{z} \rTo[l>=5em]^{h \tensor \id - \id \tensor z} C \tensor_R R\nov{z}
  \big) \ .
\end{equation}
In the mapping cone on the right we notice that while the cochain
modules are of the form $M \tensor_R R\nov{z}$ the differential is
{\it not\/} of the form $f \tensor_R \id_{R\nov{z}}$ for an
$R$\nbd-linear map~$f$; the reason is the presence of the self map
$\id \tensor z$ which ``raises the $z$\nbd-degree'', and this cannot
happen for maps of the form $f \tensor \id$.

However, if $C$ consists of finitely presented $R$\nbd-modules we can
identify the cochain complex $C \tensor_R R\nov{z}$ with $C\nov{z}$ by
Lemma~\ref{lem:series}, and we can further identify the right hand
side of~\eqref{eq:mt1} with $\totlt (D^{*,*})$ for the following
bicomplex:
\begin{equation}
  \label{eq:D}\left.
  \begin{aligned}
    D^{p,q} &= C^{p+q+1} \oplus C^{p+q} \\
    \noalign{\smallskip}
    d_h &\colon \ \, D^{p,q} \rTo D^{p+1,q} \\
    & \hphantom{\colon\ \,} (x,y) \mapsto (0, -x) \\
    \noalign{\smallskip}
    d_v & \colon \ \, D^{p,q} \rTo D^{p,q+1} \\
    & \hphantom{\colon \ \,} (x,y) \mapsto \big(-d_C(x),\,
    h(x)+d_C(y)\big) \qquad
  \end{aligned}\right\}
\end{equation}
Here $d_C$ denotes the differential of the complex~$C$. We have $d_h
\circ d_h = 0$, and the $p$th column $D^{p,*}$ of~$D^{*,*}$ is the
$p$th shift of~$\mathrm{Cone}\, (h)$, with unchanged differential, so
that $d_v \circ d_v = 0$. Finally, horizontal and vertical
differentials anti-commute: for a typical element $(x,y) \in D^{p,q} =
C^{p+q+1} \oplus C^{p+q}$ we have
\begin{multline*}
  d_v \circ d_h (x,y) = d_v (0, -x) = \big( 0, d_C(-x) \big) = -
  \big(0, d_C(x)\big) \\ = -d_h \big( (-d_C(x), h(x)+d_C(y)) \big) =
  -d_h \circ d_v (x,y) \ .
\end{multline*}
To complete the identification we postulate that the $p$th column
of~$D^{*,*}$ corresponds to the terms with coefficient~$z^p$ in the
formal \textsc{Laurent} series notation for the truncated
product~$\totlt (D^{*,*})$.

\begin{lemma}
  \label{lem:edge}
  Suppose that $C$ is a bounded above cochain complex of projective
  right $R[x,x\inv,y,y\inv]$\nbd-modules. Suppose further that $C$ is
  homotopy equivalent, as an $R$\nbd-module complex, to a bounded
  complex~$B$ of finitely generated projective right
  $R$\nbd-modules. Then the induced cochain complex $C
  \tensor_{R[x,x\inv,y,y\inv]} R[y,y\inv]\nov{x}$ is acyclic.
\end{lemma}

\begin{proof}
  Let $f \colon C \rTo B$ and $g \colon B \rTo C$ be mutually inverse
  homotopy equivalences of $R$\nbd-module complexes. There are
  $R[y,y\inv]$\nbd-module homotopy equivalences
  \[C \simeq \T(y) \simeq \T(ygf) \simeq \T(fyg) =: A\] where the
  symbol~``$y$'' denotes the $R$\nbd-linear self map given by
  multiplication by~$y$, cf.~Proposition~\ref{prop:prop_mt1}; note
  that all mapping $1$\nbd-tori here are mapping $1$\nbd-tori of
  $R$\nbd-linear maps. Now since $B$ is bounded and consists of
  finitely generated projective right $R$\nbd-modules, the complex $A
  = \T(fyg)$ is bounded and consists of finitely generated projective
  right $R[y,y\inv]$\nbd-modules.

  Let $\alpha \colon C \rTo A$ and $\beta \colon A \rTo C$ be mutually
  inverse chain homotopy equivalences of $R[y,y\inv]$\nbd-module
  complexes.  There are chain homotopy equivalences of
  $R[x,x\inv,y,y\inv]$\nbd-module complexes
  \[C \simeq \T(x) \simeq \T(x\beta\alpha) \simeq \T(\alpha x\beta)\]
  where ``$x$'' denotes the $R[y,y\inv]$\nbd-linear self map given by
  multiplication by~$x$, cf.~Proposition~\ref{prop:prop_mt1} applied
  to the ring $R[y,y\inv]$ instead of~$R$; note that all mapping
  $1$\nbd-tori here are mapping $1$\nbd-tori of
  $R[y,y\inv]$\nbd-linear maps.

  It follows that the two complexes
  \begin{equation}
    \label{eq:CE}
    C \tensor_{R[x,x\inv,y,y\inv]} R[y,y\inv]\nov{x} \quad \text{and}
    \quad \T(\alpha x \beta) \tensor_{R[x,x\inv,y,y\inv]}
    R[y,y\inv]\nov{x}
  \end{equation}
  are homotopy equivalent. Moreover, $\T(\alpha x \beta)$ is bounded
  and consists of finitely generated projective
  $R[x,x\inv,y,y\inv]$\nbd-modules by our results on~$A$. So we can
  identify the second complex in~\eqref{eq:CE} with $\totlt (D^{*,*})$
  of a certain double complex of $R[y,y\inv]$\nbd-modules. In fact, we
  are dealing with the bicomplex construction given in~\eqref{eq:D}
  above, working over the ring $R[y,y\inv]$ instead of~$R$ and applied
  to the mapping $1$\nbd-torus
    \begin{multline*}
    \T(\alpha x \beta) = \mathrm{Cone}\, \big( A \tensor_{R[y,y\inv]}
    R[x,x\inv,y,y\inv] \\ \rTo[l>=8em]^{(\alpha x \beta) \tensor \id -
      \id \tensor x} C \tensor_{R[y,y\inv]} R[x,x\inv,y,y\inv] \big) \
    .
  \end{multline*}
  Now the map $\alpha x \beta \colon A \rTo A$ is a homotopy
  equivalence so that its mapping cone is acyclic. It follows that the
  columns of~$D^{*,*}$, which are shifted copies of this mapping cone,
  are exact whence $\totlt (D^{*,*})$ is acyclic by
  Proposition~\ref{prop:double_acyclic}. In view of the homotopy
  equivalence of the complexes in~\eqref{eq:CE} this proves the Lemma.
\end{proof}

\begin{corollary}
  \label{cor:x_or_y_acyclic}
  Under the conditions of the Main Theorem part~\eqref{item:findom},
  all the cochain complexes listed under~\eqref{eq:cond_edge} are
  acyclic.
\end{corollary}

\begin{proof}
  For the complex $C \tensor_L R[y,y\inv]\nov{x}$ this is the content of
  the previous Lemma. By a change of variables, leaving $y$~fixed and
  replacing $x$ by~$x\inv$ we get the result for $C \tensor_L
  R[y,y\inv]\nov{x\inv}$. Finally, by a further change of variables
  swapping $x$ and~$y$ we cover the remaining two cases.
\end{proof}

\begin{remark}
  The change of variables $x \mapsto x\inv,\ y \mapsto y$ executed in
  the proof can be avoided by using right truncated totalisations, and
  by changing the definition of the mapping torus to involve $\id
  \tensor z\inv$ rather than $\id \tensor z$.
\end{remark}

\section{Truncated products. Two variables}
\label{sec:truncated_prods}

For a ring~$R$ we define the \textsc{Novikov} ring $R\nov{x,y} =
R\powers{x,y}[(xy)\inv]$ to be the ring of those formal \textsc{Laurent}
power series~$f$ in two variables which have the property that for
some $\ell \geq 0$ we have $x^\ell y^\ell f \in R\powers{x,y}$. That is,
\[R\nov{x,y} = \Big\{ \sum_{p,q \geq k} a_{p,q} x^p y^q \,\big|\, k \in
\bZ,\ a_{p,q} \in R \Big\} \ ,\] equipped with the obvious
multiplication of power series.

There is a related module theoretic construction, to be described
next.

\begin{definition}
  \label{def:truncprodxy}
  Let $M_{j,k}$ be a $\bZ \times \bZ$\nbd-indexed family of
  $R$\nbd-modules. We define their {\it truncated product} (more
  precisely, their {\it below-and-left truncated product\/}) by
  \[\prodblt_{p,q} M_{p,q} = \Big\{ \sum_{p,q \geq k} m_{p,q} x^p y^q
  \,\big|\, k \in \bZ,\ a_{p,q} \in M_{p,q} \Big\} \subseteq
  \prod_{p,q \in \bZ} M_{p,q} \ ,\] and if $M_{p,q} = M$ for all $p,q
  \in \bZ$ we use the notation
  \[M\nov{x,y} = \prodblt_{p,q} M \ .\]
\end{definition}

Note that $M\nov{x,y}$ has an obvious $R\nov{x,y}$\nbd-module structure
given by ``multiplication of \textsc{Laurent\/} series''.

\begin{lemma}
  \label{lem:xyseries}
  Suppose that $M$ is a finitely presented right $R$\nbd-module. There
  is a natural isomorphism
  \begin{align*}
    M \tensor_R R\nov{x,y} &\rTo^\iso M\nov{x,y} \ , \\
    m \tensor \sum_{j,k \geq \ell} r_{j,k} x^j y^k &\mapsto \quad
    \sum_{j,k \geq \ell} mr_{j,k} x^jy^k \ .
  \end{align*}
\end{lemma}

\begin{proof}
  This is similar to Lemma~\ref{lem:series}. We omit the details.
\end{proof}

\section{Below and left truncated totalisation of triple complexes}
\label{sec:blt_tot}

Let $D^{*,*,*}$ be a triple complex, cf.~\S\ref{sec:triple-complexes}.
To it we associate its {\it below-and-left truncated totalisation},
denoted $\totblt (D^{*,*,*})$; this is the cochain complex given by
\[\big( \totblt D^{*,*,*} \big)^n = \prodblt_{p,q} D^{p,q,n-p-q}\]
with differential $d = d_x + d_y +d_z$. Explicitly,
\[ d \colon \sum_{p,q \geq k} m_{p,q} x^p y^q \mapsto \sum_{p,q \geq
  k} \big( d_x (m_{p-1,q}) + d_y (m_{p,q-1}) + d_z (m_{p,q}) \big) x^p
y^q\] where $m_{p,q} = 0$ if $p<k$ or $q<k$.

\begin{proposition}
  \label{prop:triple_exact}
  Suppose the triple complex $D^{*,*,*}$ is exact in
  $z$\nbd-direc\-tion; that is, suppose that the chain complexes
  $D^{x,y,*}$ are acyclic for all $x, y \in \bZ$. Then $\totblt
  (D^{*,*,*})$ is acyclic.
\end{proposition}

\begin{proof}
  This is elementary: we prove directly that any cocycle in the
  totalisation is a coboundary. Let $m = \sum_{p,q \geq k} m_{p,q} x^p
  y^q \in \big( \totblt D^{*,*,*} \big)^n$ be such that $d(m) = 0$. By
  definition of~$d$ this means that $d_z (m_{k,k}) = 0$; by exactness
  in $z$\nbd-direction we find an element $b_{k,k} \in D^{k,k,n-1-2k}$
  with $d_z(b_{k,k}) = m_{k,k}$. Now iteratively for $\ell = 1,\, 2,\,
  3,\, \cdots$ we note that
  \[
  \begin{split}
    d_z \big( m_{k+\ell, k} - d_x (b_{k+\ell-1, k}) \big) = d_z
    (m_{k+\ell,k}) + d_x d_z (b_{k+\ell-1, k}) \qquad \qquad \\ = d_z
    (m_{k+\ell,k}) + d_x (m_{k+\ell-1, k}) = 0 \ ,
  \end{split}
  \] using that $d_x d_z = - d_z d_x$ and $d(m) = 0$. By exactness in
  $z$\nbd-direction we find $b_{k+\ell,k}$ with $d_z (b_{k+\ell,k})
  = m_{k+\ell, k} - d_x (b_{k+\ell-1, k})$ so that
  \[d_x (b_{k+\ell-1, k}) + d_z (b_{k+\ell, k}) = m_{k+\ell, k} \ .\]
  Similarly, we note that
  \[
  \begin{split}
    d_z \big( m_{k, k+\ell} - d_y (b_{k, k+\ell-1}) \big) = d_z
    (m_{k, k+\ell}) + d_y d_z (b_{k, k+\ell-1}) \qquad \qquad \\ = d_z
    (m_{k, k+\ell}) + d_y (m_{k, k+\ell-1}) = 0 \ ,
  \end{split}
  \] using that $d_y d_z = - d_z d_y$ and $d(m) = 0$. By exactness in
  $z$\nbd-direction we find $b_{k, k+\ell}$ with $d_z (b_{k, k+\ell})
  = m_{k, k+\ell} - d_y (b_{k, k+\ell-1})$ so that
  \[d_y (b_{k, k+\ell-1}) + d_z (b_{k, k+\ell}) = m_{k, k+\ell} \
  .\]

  So far we have constructed elements of the form $b_{*,k}$ and
  $b_{k,*}$, and proceed now to iterate the construction for $j=k+1,\,
  k+2,\,\cdots$ as follows. First, $d(m) = 0$ implies that $d_x
  (m_{j-1,j}) + d_y (m_{j,j-1}) + d_z (m_{j,j}) = 0$. By construction
  of the previous $b_{p,q}$ we have $m_{j-1,j} = d_x (b_{j-2,j}) + d_y
  (b_{j-1,j-1}) + d_z (b_{j-1,j})$ so that
  \[d_x (m_{j-1,j}) = d_x d_y (b_{j-1,j-1}) -d_z d_x (b_{j-1,j})\]
  and, in the same way,
  \[d_y (m_{j,j-1}) = -d_x d_y (b_{j-1,j-1}) - d_z d_y (b_{j,j-1})\]
  which together results in
  \[d_z \big( m_{j,j} - d_x (b_{j-1,j}) - d_y (b_{j,j-1}) \big) = 0 \ .\]
  (We have used the usual convention that $b_{p,q} = 0$ if $p<k$ or
  $q<k$.) By exactness in $z$\nbd-direction there is $b_{j,j}$ with
  $d_z (b_{j,j}) = m_{j,j} - d_x (b_{j-1,j}) - d_y (b_{j,j-1})$ or,
  re-arranged,
  \[d_x (b_{j-1,j}) + d_y (b_{j,j-1}) + d_z (b_{j,j}) = m_{j,j} \ .\]
  Still keeping $j$ fixed, we now iterate over $\ell = 1,\, 2,\,
  \cdots$ by observing that
  \[d_z \big( m_{j+\ell, j} - d_x (b_{j+\ell-1, j}) - d_y (b_{j+\ell,
    j-1}) \big) = 0 \] so we find $b_{j+\ell, j}$ with $d_z
  (b_{j+\ell, j}) = m_{j+\ell, j} - d_x (b_{j+\ell-1, j}) - d_y (b_{j+\ell,
    j-1})$ or, re-arranged,
  \[d_x (b_{j+\ell-1, j}) + d_y (b_{j+\ell, j-1}) + d_z  (b_{j+\ell,
    j}) = m_{j+\ell, j} \ .\] Again, for the same~$j$ we find in a
  similar manner an element $b_{j, j+\ell}$ with
  \[d_x (b_{j-1, j+\ell}) + d_y (b_{j, j+\ell-1}) + d_z (b_{j,
    j+\ell}) = m_{j, j+\ell} \ .\]

  This finishes both iterations. It remains to observe that, by
  construction, we have shown that
  \[d \Big( \sum_{p,q \geq k} b_{p,q} x^p y^q \Big) = \sum_{p,q \geq
    k} m_{p,q} x^p y^q = m\] so that $m$~is a coboundary as claimed.
\end{proof}

\section{The mapping $2$-torus}
\label{sec:2torus}

Suppose we have a cochain complex~$C$ of $R$\nbd-modules with
differential denoted~$d_C$, and two self maps $f,g \colon C \rTo C$
which commute up to homotopy. Suppose moreover we are given a specific
choice of homotopy $H \colon fg \simeq gf$ such that $d_C H + H d_C =
fg - gf$. To such data we associate a cochain complex of
$\LL$\nbd-modules, the {\it mapping $2$\nbd-torus\/} $\T(f,g; H)$. The
module in degree~$n$ is the direct sum
\begin{equation}
  \label{eq:def_TfgH}
  \begin{split}
    \T(f,g;H)^n = C^{n+2} \tensor_R \LL \, \oplus\, \big(C^{n+1} \tensor_R
    \LL \qquad \quad \\ \qquad \qquad \oplus\, C^{n+1} \tensor_R \LL \big) \,\oplus\, C^n
    \tensor_R \LL \ ,
  \end{split}
\end{equation}
and the coboundary map $d_\T \colon \T(f,g;H)^n \rTo \T(f,g;H)^{n+1}$ is
given by the following matrix:
\begin{equation}
  \label{eq:def_d_T}
  d_\T =
  \begin{pmatrix}
    d_C \tensor \id & 0 & 0 & 0 \\
    \noalign{\smallskip}
    -(g \tensor \id - \id \tensor x) & -d_C \tensor \id & 0 & 0 \\
    \noalign{\smallskip}
    f \tensor \id - \id \tensor y  & 0 & -d_C \tensor \id & 0 \\
    \noalign{\smallskip}
    H \tensor \id & f \tensor \id - \id \tensor y & g \tensor \id -
    \id \tensor x &  d_C \tensor \id
  \end{pmatrix}
\end{equation}
By construction, $\T(f,g;H)$ is a lower $4$-triangular complex in the
sense of Definition~\ref{def:triangle_complexes}.

\begin{remark}
  If $fg=gf$ we may choose $H=0$, and the mapping $2$\nbd-torus
  $\T(f,g;0)$ is nothing but the total complex of the following twofold
  cochain complex concentrated in columns $-2$, $-1$ and~$0$:
  {\small\begin{equation}
    \label{eq:def_Tfg2}
    C \tensor_R \LL \rTo^\alpha
    \begin{matrix}
      C \tensor\limits_R \LL \\ \oplus \\ C \tensor\limits_R^{\vphantom{e}}\LL
    \end{matrix}
    \rTo^\beta C \tensor_R \LL
  \end{equation}}%
  Here the maps $\alpha$ and~$\beta$ are given by the matrices
  \begin{multline*}
    \alpha =
    \begin{pmatrix}
      -\big( g \tensor \id - \id \tensor x \big) \\ f \tensor \id - \id
      \tensor y
    \end{pmatrix}
    \quad \text{and}\\
    \beta =
    \begin{pmatrix}
      f \tensor \id - \id \tensor y\,, & g \tensor \id - \id \tensor x
    \end{pmatrix}
    \ . \qquad\qquad\qquad
  \end{multline*}
  (This twofold cochain complex features {\it commuting\/}
  differentials and thus should be converted to a double complex by
  changing the ``vertical'' differential coming from~$C$ in the middle
  summand by a sign.)
\end{remark}

\subsection*{\hskip 1.8em Properties of the mapping $2$-torus}

For later applications we need to record that the mapping
$2$\nbd-torus has properties analogous to those of the mapping
$1$\nbd-torus of~\S\ref{sec:1-tori}. We keep the notation from above.
\begin{theorem}
  \label{thm:properties}
  \begin{enumerate}[{\rm (1)}]
  \item \label{item:h} Suppose that the maps $f$and~$g$ commute, and
    suppose that we are given a self map $h \colon C \rTo C$ together
    with a homotopy $A \colon h \simeq \id$ so that $d_C A +
    A d_C = h - \id$. Then $h(fAg-gAf)$ is a homotopy from
    $(hf)(hg)$ to~$(hg)(hf)$, and the matrix
    \[\qquad \Phi =
    \begin{pmatrix}
      h \tensor \id & 0 & 0 & 0 \\
      -hgA \tensor \id & h \tensor \id & 0 & 0 \\
       hfA \tensor \id & 0 & h \tensor \id & 0 \\
      h(fAg - gAf)A \tensor \id & -hfA \tensor \id & -hgA \tensor \id &
      h \tensor \id
    \end{pmatrix}\]
    defines a quasi-isomorphism
    \[\Phi \colon \T(f,g;0) \rTo \T\big( hf, hg;
    h(fAg-gAf)\big) \ .\] If $C$ is a bounded above complex of
    projective modules over~$R$ then $\T(f,g;0)$ and~$\T\big(hf,hg;
    h(fAg-gAf)\big)$ are homotopy equivalent.
  \item \label{item:mather} Suppose that the maps $f$and~$g$ commute
    as before. Given cochain maps $\alpha \colon B \rTo C$ and $\beta
    \colon C \rTo B$ and a homotopy $A \colon \alpha\beta \simeq \id$
    so that $d_C A + A d_C = \alpha\beta - \id$, the composite map
    $\beta (fAg - gAf) \alpha$ is a homotopy from $(\beta f
    \alpha)(\beta g \alpha)$ to $(\beta g \alpha)(\beta f \alpha)$,
    the composite map $\alpha \beta(fAg - gAf)$ is a homotopy from
    $(\alpha \beta f) (\alpha \beta g)$ to $(\alpha \beta g) (\alpha
    \beta f)$, and the diagonal matrix with entries $\alpha \tensor
    \id$ on the main diagonal defines a cochain map
    \begin{multline*}
      \qquad \qquad \qquad \qquad \alpha_* \colon \T \big( \beta f
      \alpha, \beta g \alpha; \beta (fAg - gAf) \alpha \big) \\ \rTo
      \T \big( \alpha\beta f, \alpha\beta g; \alpha\beta (fAg - gAf)
      \big) \ .
    \end{multline*}
    If $\alpha$~is a quasi-isomorphism so is~$\alpha_*$. If in
    addition $B$ and~$C$ are bounded above complexes of projective
    $R$\nbd-modules then $\alpha_*$~is a homotopy equivalence.
  \item \label{item:L} If $C$ is actually a complex of
    $\LL$\nbd-modules, then there is an $\LL$\nbd-linear
    quasi-isomorphism $\T(y,x;0) \rTo C$. If $C$ is a bounded above
    complex of projective modules over the \textsc{Laurent} ring $\LL$
    then $C$ and~$\T(y,x;0)$ are homotopy equivalent.
  \end{enumerate}
\end{theorem}

\begin{proof}
  \eqref{item:h} First we calculate
  \begin{multline*}
    d_C h(fAg-gAf) + h(fAg-gAf) d_C = hf (d_C A
    + A d_C) g - hg (d_C A
    + A d_C)f \\ = hf(h-\id)g - hg(h-\id)f = hfhg - hghf
  \end{multline*}
  (recall that $fg=gf$) so that $h(fAg-gAf)$ is a homotopy from
  $(hf)(hg)$ to~$(hg)(hf)$ as claimed. Next, we need to check that
  $\Phi$~defines a cochain map. Let $\hat d_\T$ denote the differential
  of $\T\big( hf, hg; h(fAg-gAf)\big)$; it is given by a matrix similar
  to~\eqref{eq:def_d_T} with $f$ and~$g$ replaced by~$hf$ and~$hg$,
  respectively, and the homotopy~$H$ replaced by $h(fAg-gAf)$. We need
  to verify that $\hat d_\T \Phi = \Phi d_\T$. The calculation is
  tedious but straightforward; we concentrate on the first entry in
  the fourth row of the product, leaving the others as exercises to
  the reader. In the case at hand we have to show that the two sums
  \[\begin{split}
   \big( \Phi d_\T \big)_{4,1} = h(fAg-gAf)Ad_C \tensor \id 
    + hfAg \tensor \id - hfA \tensor x
    \qquad\qquad \\
    - hgAf \tensor \id + hgA \tensor y
  \end{split}\] on the one hand, and
  \[\begin{split}
    \big(\hat d_\T \Phi \big)_{1,4} = h(fAg-gAf)h \tensor \id - hfhgA
    \tensor \id + hgA \tensor y \qquad\qquad \\
    + hghfA \tensor \id - hfA \tensor x + d_C h(fAg-gAf)A \tensor \id
  \end{split}\] on the other hand, are equal. By cancelling equal
  terms and re-arranging, remembering that $f$, $g$ and~$h$ commute
  with~$d_C$, this amounts to verifying the equality
  \[
  \begin{split}
    hfAgAd_C - hfd_CAgA + hgd_CAfA -
    hgAfAd_C
    \qquad \qquad\qquad \\
    = hgAf -hfAg -hfhgA +hghfA + hfAgh - hgAfh \ .
  \end{split}\] Since $f$ and~$g$ are cochain maps we can now add the
  zero-term
  \[hfA(gd_C - d_Cg)A + hgA(d_C f -
  f d_C)A\] on the left hand side, and collect terms on the
  right, to get the equivalent equation
  \begin{multline*}
    hfAg(Ad_C + d_CA) - hf(A d_C + d_C
    A)gA \\ 
    + hg(d_CA + Ad_C)fA - hgAf(Ad_C
    + d_CA) \\
    = hfAg(h-\id) - hgAf (h - \id) - (hfhg - hghf)A
  \end{multline*}
  which is satisfied since $Ad_C + d_CA = h - \id$ and
  $fg=gf$.

  Now $\Phi$ is a lower triangular matrix which has quasi-isomorphisms
  on its diagonal, and both the differentials~$d_\T$ and $\hat d_\T$
  are lower triangular as well. It follows from
  Lemma~\ref{lem:triangle_maps} that $\Phi$~is a
  quasi-isomorphism. --- If $C$ is a bounded above complex of
  projective $R$\nbd-modules, then both source and target of~$\Phi$
  are bounded above complexes of projective $\LL$-modules, and it
  is well known that a quasi-isomorphism of such complexes must be a
  homotopy equivalence.

  \eqref{item:mather} This is similar to the proof of
  part~\eqref{item:h}. We omit the details.

  \eqref{item:L} Abbreviate $\LL$ by~$L$. The map
  \[\gamma \colon C \tensor_R L \rTo C \ , \quad z \tensor p \mapsto
  zp\] induces a cochain complex map $\T(y,x;0) \rTo C$. To show that
  this is a quasi-isomorphism it is enough to verify that its mapping
  cone is acyclic; but this mapping cone is, up to isomorphism, the 
  total complex of the double complex
  \begin{equation*}
    C \tensor_R L \rTo^\alpha
    \begin{matrix}
      C \tensor\limits_R L \\ \oplus \\ C \tensor\limits_R^{\vphantom{e}}L
    \end{matrix}
    \rTo^\beta C \tensor_R L \rTo^\gamma C
  \end{equation*}
  with~$C$ sitting in column $p=1$, cf.~\eqref{eq:def_Tfg}, so that it
  is sufficient (using Lemma~\ref{lemma:dblcpx}) to demonstrate that
  this double complex has exact rows.

  Since $L$ is a free $R$-module, an element $z\in C\otimes_R L$ can
  be expressed uniquely as
  \[ z=\sum_{i,j \in \bZ} m_{i,j}\otimes x^iy^j\ ,\] with $m_{i,j} =
  0$ for almost all pairs~$(i,j)$. We say that {\it $z$~has
    $x$\nbd-amplitude in the interval~$[a,b]$\/} if $m_{i,j} =0$ for
  $i \notin [a,b]$, and that {\it $z$~has $y$\nbd-amplitude in the
    interval~$[a,b]$\/} if $m_{i,j} =0$ for $j \notin [a,b]$. The {\it
    support of~$z$\/} is the set of all pairs~$(i,j)$ with $m_{i,j}
  \neq 0$.

  \smallskip

  To show that $\alpha$ is injective we note that
  \[\alpha(z)=\begin{pmatrix}
    \displaystyle\sum_{i,j \in \bZ} \big(m_{i,j}x\otimes x^iy^j -
    m_{i,j}\otimes x^{i+1}y^j\big) \\
    \displaystyle\sum_{i,j \in \bZ} \big(-m_{i,j}y\otimes x^iy^j +
    m_{i,j}\otimes x^iy^{j+1}\big)
  \end{pmatrix} \ ,\] so $\alpha(z)=0$ implies in particular
  \begin{equation}
    \label{eq:is0}
    m_{i,j}x-m_{i-1,j}=0
  \end{equation}
  for all pairs $(i,j)$.  The support of $z$, if non-empty, is well
  ordered by the order $(i,j)<(k,l)$ whenever $j<l$, or $j=l$ and
  $i<k$. When $(i,j)$ is the element of the support which is minimal
  with respect to this order we have $m_{i,j} \neq 0$, by definition
  of support, and $m_{i-1,j} = 0$. Using~\eqref{eq:is0} we have
  $m_{i,j}x=0$, and since $x$~is a non zero divisor we arrive at the
  contradiction $m_{i,j}=0$. We conclude that $\alpha(z)=0$ forces $z
  = 0$.

  \smallskip

  To show that $\mathrm{im}\,\alpha = \ker\beta$ we will take an
  element~$(z_1,z_2)$ of $\ker\beta$ and show that it can be reduced
  to~$0$ by subtracting a sequence of elements of $\mathrm{im}\,\alpha
  \subseteq \ker\beta$. This clearly implies that $(z_1,z_2) \in
  \mathrm{im}\,\alpha$ as required.

  So let $(z_1,z_2) \in \ker\beta$, where
  \[z_1 = \sum_{i,j \in \bZ} m_{i,j}\otimes x^iy^j \qquad \text{and}
  \qquad z_2 = \sum_{i,j \in \bZ} n_{i,j}\otimes x^iy^j\ .\]
  Choose integers $a \leq b$ such that both $z_1$ and~$z_2$ have
  $x$\nbd-amplitude in~$[a,b]$.

  If $a<b$, we define
  \[u = \sum_{j \in \bZ} m_{b,j}\otimes x^{b-1} y^{j} \in C \tensor_R
  L\] and set $(z_1',z_2') = (z_1,z_2)+\alpha(u)$. Then $z_1'$ has
  $x$\nbd-amplitude in $[a,b-1]$ while the $x$-amplitude of $z_2'$ is
  in $[a,b]$. The element $(z_1',z_2')$ will be in
  $\mathrm{im}\,\alpha$ if and only if $(z_1,z_2) \in
  \mathrm{im}\,\alpha$.

  By iteration, we may thus assume that our initial pair $(z_1,z_2)$
  is such that the $x$\nbd-amplitude of~$z_1$ is in $\{a\} = [a,a]$,
  and $z_2$~has $x$\nbd-amplitude in $[a,b]$ for some $b \geq a$.
  Note that then $\beta\big((z_1,0)\big)$ will also have $x$-amplitude
  in $\{a\}$, and that $\beta\big((z_1,z_2)\big)$ will have
  $x$\nbd-amplitude in~$[a,b+1]$.

  Write $\beta\big((z_1,z_2)\big) = \sum_{i=a}^{b+1}\sum_{j \in \bZ}
  r_{i,j}\otimes x^{i}y^j$. Then $r_{i+1,j}=n_{i,j}$ for every $i \geq
  a$ as $m_{i+1,\ell} = 0$ for any~$\ell$.  Since
  $\beta\big((z_1,z_2)\big) = 0$ we must have $n_{i,j} = r_{i+1,j} =
  0$ for every $i \geq a$, so that in fact $z_2 = 0$. This in turn
  implies that $\beta\big((z_1,0)\big) = 0$ so that $r_{a,\ell+1} =
  m_{a,\ell} - m_{a,\ell+1}=0$. If $z_1$ is non-zero, we let $\ell$ be
  maximal with $m_{a, \ell} \neq 0$. Then $m_{a,\ell+1} = 0$ and
  consequently $0 = r_{a,\ell+1} = m_{a,\ell}$, a contradiction. We
  conclude that $(z_1,z_2) = 0 \in \mathrm{im}\,\alpha$ as required.

  \smallskip

  To show $\mathrm{im}\,\beta = \ker\gamma$ we proceed in a similar
  manner. For a given $z\in \mathrm{ker}\gamma$ may be converted, by
  subtracting elements of $\mathrm{im}\,\beta$, to an element of the
  form $m_{a,a} \tensor x^ay^a$ which lies in $\mathrm{im}\,\beta$ if
  and only if $z$~does. The conversion involves a systematic reduction
  of both the $x$\nbd-amplitude and the $y$\nbd-amplitude; we omit the
  details. Now $\gamma(z)=0$ implies $m_{a,a} x^ay^a = \gamma (m_{a,a}
  \tensor x^ay^a) = 0$, and since neither~$x$ nor~$y$ is a zero
  divisor this gives us $m_{a,a} \tensor x^a y^a = 0 \in \mathrm{im}\,
  \beta$.

  \smallskip

  Finally, let us remark that $\gamma$ is surjective since
  $\gamma(p\otimes1)=p$.
\end{proof}

\subsection*{\hskip 1.8em The mapping $2$-torus analogue}

For later use we record a construction somewhat similar to the mapping
$2$\nbd-torus. Suppose we have a cochain complex~$C$ of
$R$\nbd-modules with differential denoted~$d_C$, and two self maps
$f,g \colon C \rTo C$ which commute up to homotopy. Suppose moreover
we are given a specific choice of homotopy $H \colon fg \simeq gf$
such that $d_C H + H d_C = fg - gf$. To such data we associate a
cochain complex of $R$\nbd-modules, the {\it mapping $2$\nbd-torus
  analogue\/} $\A(f,g; H)$. The module in degree~$n$ is the direct sum
\begin{equation}
  \label{eq:def_AfgH}
  \A(f,g;H)^n = C^{n+2} \, \oplus\, \big(C^{n+1} \oplus\, C^{n+1}
  \big) \,\oplus\, C^n\ ,
\end{equation}
and the boundary map $d_\A \colon \A(f,g;H)^n \rTo \A(f,g;H)^{n+1}$ is
given by the matrix
\begin{equation}
  \label{eq:def_d_A}
  d_\A =
  \begin{pmatrix}
    d_C & 0 & 0 & 0 \\
    \noalign{\smallskip}
    -g & -d_C& 0 & 0 \\
    \noalign{\smallskip}
    f & 0 & -d_C & 0 \\
    \noalign{\smallskip}
    H & f & g &  d_C
  \end{pmatrix} \ .
\end{equation}

\begin{remark}
  \label{rem:A}
  \begin{enumerate}[{\rm (a)}]
  \item   If $fg=gf$ we may choose $H=0$, and the mapping $2$\nbd-torus analogue
    $\A(f,g;0)$ is nothing but the total complex of the twofold
    chain complex 
    \begin{equation}
      \label{eq:def_Tfg}
      C \rTo^\alpha C \oplus C \rTo^\beta C
    \end{equation}
    concentrated in columns $-2$, $-1$ and~$0$.
    Here the maps $\alpha$ and~$\beta$ are given by the matrices
    \[\alpha =
    \begin{pmatrix}
      -g \\ f 
    \end{pmatrix}
    \quad \text{and} \quad
    \beta =
    \begin{pmatrix}
      f \,, & g 
    \end{pmatrix}
    \ .\]
  \item \label{item:acyclicA} Suppose in addition that one of $f$
    or~$g$ is a quasi-isomorphism. Then $\A(f,g;0)$ is acyclic. For
    $\A(f,g;0)$ can be obtained by taking iterated mapping cones in
    a commuting square of cochain complexes:
    \begin{diagram}
      C & \rTo^f & C \\ \uTo<g && \uTo<g \\ C & \rTo^f & C
    \end{diagram}
    That is, $\A(f,g;0)$ is isomorphic to the cochain complex obtained
    by applying the algebraic mapping cone functor to the horizontal
    maps in the square first, and then again to the resulting
    ``vertical'' map (or, by symmetry, with roles of horizontal and
    vertical exchanged). But the mapping cone of a quasi-isomorphism
    is acyclic, as is the mapping cone of a map of acyclic complexes,
    whence the assertion.
  \end{enumerate}
\end{remark}

\begin{lemma}
  \label{lem:properties_A}
  \begin{enumerate}[{\rm (1)}]
  \item \label{item:hA} Suppose that the maps $f$and~$g$ commute, and
    suppose that we are given a self map $h \colon C \rTo C$ together
    with a homotopy $A \colon h \simeq \id$ so that $d_C A + A d_C = h
    - \id$. Then $h(fAg-gAf)$ is a homotopy from $(hf)(hg)$
    to~$(hg)(hf)$, and the matrix
    \[\qquad
    \begin{pmatrix}
      h & 0 & 0 & 0 \\
      -hgA & h & 0 & 0 \\
       hfA & 0 & h & 0 \\
      h(fAg-gAf)A & -hfA & -hgA & h
    \end{pmatrix}\] defines a quasi-isomorphism
    \[\A(f,g;0) \rTo \A\big( hf, hg; h(fAg-gAf)\big) \ .\] If $C$ is a
    bounded above complex of projective modules over~$R$ then
    $\A(f,g;0)$ and~$\A\big(hf,hg; h(fAg-gAf)\big)$ are homotopy
    equivalent.
  \item \label{item:matherA} Suppose that the maps $f$and~$g$ commute
    as before. Given cochain maps $\alpha \colon B \rTo C$ and $\beta
    \colon C \rTo B$ and a homotopy $A \colon \alpha\beta \simeq \id$
    so that $d_C A + A d_C = \alpha\beta - \id$, the map $\beta (fAg -
    gAf) \alpha$ is a homotopy from $(\beta f \alpha)(\beta g \alpha)$
    to $(\beta g \alpha)(\beta f \alpha)$, the map $\alpha\beta (fAg -
    gAf)$ is a homotopy from $(\alpha\beta f) (\alpha\beta g)$ to
    $(\alpha\beta g) (\alpha\beta f)$, and the diagonal matrix with
    entries $\alpha$ on the diagonal defines a cochain complex map
    \begin{multline*}
      \qquad \qquad \qquad \qquad \alpha_* \colon \A\big( \beta f
      \alpha, \beta g \alpha; \beta (fAg - gAf) \alpha \big) \\ \rTo
      \A\big( \alpha\beta f, \alpha\beta g; \alpha\beta (fAg - gAf)
      \big) \ .
    \end{multline*}
    If $\alpha$~is a quasi-isomorphism so is~$\alpha_*$. If in
    addition $B$ and~$C$ are bounded above complexes of projective
    $R$\nbd-modules then $\alpha_*$~is a homotopy equivalence.
  \end{enumerate}
\end{lemma}

\begin{proof}
  This is almost identical to the proof of
  Theorem~\ref{thm:properties}, but slightly easier due to the absence
  of tensor products. We omit the details.
\end{proof}

\section{Mapping $2$-tori and totalisation. Applications}
\label{sec:((x,y))}

Now suppose that $C$ is an $R$\nbd-finitely dominated bounded above
cochain complex of projective $\LL$\nbd-modules. We will prove now
that $C \tensor_\LL R\nov{x,y}$~is acyclic.

By our hypothesis there exists a bounded cochain complex~$B$ of
finitely generated projective $R$\nbd-modules together with mutually
inverse homotopy equivalences
\[\alpha \colon B \rTo C \quad \text{and} \quad \beta \colon C \rTo
B\] of $R$\nbd-module cochain complexes. Choose a homotopy $A \colon
\alpha \beta \simeq \id_C$. By Theorem~\ref{thm:properties} there are
$\LL$\nbd-linear homotopy equivalences
\[
\begin{split}
  C \lTo \T(y,x;0) \rTo \T\big(\alpha\beta y, \alpha\beta x;
  \alpha\beta (yAx-xAy)\big) \qquad \qquad \qquad \\
  \lTo^{\alpha_*} \T\big(\beta y \alpha, \beta x \alpha; \beta
  (yAx-xAy) \alpha\big) =: Z
\end{split}\] so that $C \tensor_\LL R\nov{x,y} \simeq Z \tensor_\LL
R\nov{x,y}$.

Now note that we have isomorphisms
\[\begin{split}B \tensor_R \LL \tensor_\LL R\nov{x,y} \qquad
  \qquad \qquad \\ \iso B \tensor_R R\nov{x,y} \iso B\nov{x,y} \
  ,\end{split}\] the second one being due to the fact that
$B$~consists of finitely generated projective $R$\nbd-modules, so that
Lemma~\ref{lem:xyseries} applies. We can thus identify the $n$th
cochain module of $Z \tensor_\LL R\nov{x,y}$ with
\begin{equation}
  \label{eq:identification}
  B^{n+2} \nov{x,y} \oplus \big( B^{n+1} \nov{x,y} \oplus B^{n+1}
  \nov{x,y}\, \big) \oplus B^n \nov{x,y} \ ,
\end{equation}
and the differential with the matrix
\begin{equation*}
  \begin{pmatrix}
    d_B \nov{x,y} & 0 & 0 & 0 \\ \noalign{\smallskip}
    - \big( \beta x \alpha \nov{x,y} - x \big) & -d_B\nov{x,y} & 0 & 0 \\ \noalign{\smallskip}
    \beta y \alpha \nov{x,y} - y & 0 & -d_B\nov{x,y} & 0 \\
    \noalign{\smallskip}
    \beta(yAx-xAy)\alpha \nov{x,y} & \beta y \alpha \nov{x,y} -
    y & \beta x \alpha \nov{x,y} - x &  d_B\nov{x,y}
  \end{pmatrix} \ .
\end{equation*}
(The symbol $d_B$ denotes the differential of~$B$. Note that the
letter ``$x$'' in the term $\beta x \alpha$ denotes a self map of~$C$,
while the symbol $x$ by itself denotes a self map of $B\nov{x,y}$;
similarly with $y$ in place of~$x$.)

This cochain complex arises as the realisation of a triple
complex. Indeed, for $x,y,z \in \bZ$ let
\[T^{x,y,z} = B^{x+y+z+2} \oplus B^{x+y+z+1} \oplus B^{x+y+z+1} \oplus
B^{x+y+z} \,\] and define differentials
\begin{align*}
  d_x & \colon T^{x,y,z} \rTo T^{x+1,y,z} \ , & (r,\,s,\,t,\,u)
  & \mapsto (0,\, r,\, 0,\, -t) \\
  d_y & \colon T^{x,y,z} \rTo T^{x,y+1,z} \ , & (r,\,s,\,t,\,u)
  & \mapsto (0,\, 0,\, -r,\, -s \big)
\end{align*}
as well as
\[d_z =
\begin{pmatrix}
  d_B & 0 & 0 & 0 \\ \noalign{\smallskip}
  -\beta x \alpha & -d_B & 0 & 0 \\ \noalign{\smallskip}
   \beta y \alpha & 0 & -d_B & 0 \\
  \noalign{\smallskip}
  H & \beta y \alpha & \beta x \alpha &  d_B
\end{pmatrix}
\colon T^{x,y,z} \rTo T^{x,y,z+1}\] where $H = \beta
(yAx-xAy)\alpha$. That is, $T^{x,y,*}$ is $\A (\beta
y \alpha, \beta x \alpha, H)$ shifted down $x+y$ times. --- With these
definitions $\totblt (T^{*,*,*})$ is precisely the chain complex
\[Z \tensor_\LL R\nov{x, y}\] under the identification made
above~\eqref{eq:identification}. Now the complex $\A(\beta y \alpha,
\beta x \alpha; H)$ is acyclic by Lemma~\ref{lem:properties_A} and
Remark~\ref{rem:A}~\eqref{item:acyclicA}; in more detail, we have a
chain of homotopy equivalences
\[\A (\beta y \alpha, \beta x \alpha; H) \rTo^\simeq \A
\big(\alpha\beta y, \alpha\beta x; \alpha\beta (yAx - xAy) \big)
\lTo^\simeq \A (y,x;0)\] with the cochain complex on the right being
acyclic. We conclude from Proposition~\ref{prop:triple_exact} that $Z
\tensor_\LL R\nov{x, y}$ is acyclic.

\begin{corollary}
  \label{cor:XYacyclic}
  If $C$ is $R$\nbd-finitely dominated, all four of the chain
  complexes listed in~\eqref{eq:cond_vertex} of Theorem~\ref{thm:main}
  are acyclic.
\end{corollary}

\begin{proof}
  For $C \tensor_R R\nov{x,y}$ this has just been shown. The other cases
  follow by a ``change of coordinates'' replacing $x$ by~$x\inv$
  and/or $y$ by~$y\inv$.
\end{proof}

\goodbreak

\part{Acyclicity implies finite domination}
\label{part:2}

We begin by giving a quick proof of the second implication. Suppose
that the chain complexes listed in~\eqref{eq:cond_edge}
and~\eqref{eq:cond_vertex} are acyclic. We will use
Theorem~\ref{thm:schuetz}, applied to the case $G = \bZ^2$ and $N =
0$, to conclude that $C$~is $R$\nbd-finitely dominated. A non-trivial
character $\chi \colon G \rTo \bR$ can be identified with a non-zero
vector $\chi = (a,b) \in \bR^2$, using the standard inner product
of~$\bR^2$. If $\chi = (a,0)$, for $a > 0$, then $\widehat{RG_\chi} =
R[y,y\inv]\nov{x}$ so that acyclicity in~\eqref{eq:cond_edge} implies
acyclicity of $C \tensor_{R[G]} \widehat{RG_\chi}$. The situation is
similar for $a<0$ with $\widehat{RG_\chi} = R[y,y\inv]\nov{x\inv}$, or
for $\chi = (0,b)$ with $b \neq 0$ in which case $\widehat{RG_\chi} =
R[x,x\inv]\nov{y^{\text{sgn}\, b}}$.

For $\chi = (a,b)$ with $a,b \neq 0$, acyclicity
in~\eqref{eq:cond_vertex} implies acyclicity of $C \tensor_{R[G]}
\widehat{RG_\chi}$. For example, if $a,b > 0$ then $\widehat{RG_\chi}$
contains $R\nov{x,y}$ so that
\[C \tensor_{R[G]} \widehat{RG_\chi} \iso C \tensor_{R[G]} R\nov{x,y}
\tensor_{R\nov{x,y}} \widehat{RG_\chi} \simeq 0\] as the chain complexes
in~\eqref{eq:cond_vertex} are actually contractible. In general, for
$a,b \neq 0$ we have $\widehat{RG_\chi} \supset R\nov{x^{\text{sgn}\,
    a}, y^{\text{sgn}\, b}}$, and the argument applies {\it mutatis
  mutandis}.

\medbreak

We will nevertheless give a different and new proof of the implication
in the remainder of the paper. We eschew the use of controlled algebra
and inductive arguments in favour of standard homological algebra
combined with homotopy-theoretic arguments. While not short our proof
is conceptually simple, and lends itself to generalisations for
\textsc{Laurent} rings with many indeterminates. These generalisations
can be done at the expense of introducing more elaborate
combinatorics; this topic will be taken up in a separate paper.

\section{Diagrams indexed by the face lattice of a square}
\label{sec:diagrams}

Let $S$ denote the square $[-1,1]^2$, a convex polytope in $\bR^2$. In
what follows a face of~$S$ will always assumed to be non-empty, unless
specified otherwise. We will orient the edges of~$S$, and~$S$ itself,
counter-clockwise as indicated in the following picture which also
shows our labelling for the faces of~$S$:
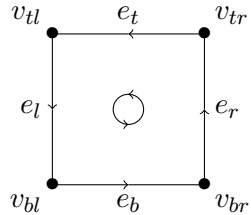
\begin{center}
\begin{tikzpicture}
  \draw[->] (0,0) node {\(\bullet\)} node[below left] {\(v_{bl}\)} --  (1,0) node[below] {\(e_b\)}; \draw (1,0) -- (2,0);
  \draw[->] (2,0) node {\(\bullet\)} node[below right] {\(v_{br}\)}  --  (2,1) node[right] {\(e_r\)}; \draw (2,1) -- (2,2);
  \draw[->] (2,2) node {\(\bullet\)} node[above right] {\(v_{tr}\)}  --  (1,2) node[above] {\(e_t\)}; \draw (1,2) -- (0,2);
  \draw[->] (0,2) node {\(\bullet\)} node[above left] {\(v_{tl}\)}  --  (0,1) node[left] {\(e_l\)}; \draw (0,1) -- (0,0);
  \draw[decoration={markings, mark=at position 0.25 with \arrow{>},
  mark=at position 0.75 with {\arrow{>}}}, postaction={decorate}] (1,1) circle (0.2);
\end{tikzpicture}
\end{center}
The orientation gives us a choice of incidence numbers $[F:G] \in
\{-1,\, 0,\, 1\}$ for all faces $F$ and~$G$ of~$S$. For example,
$[e:S] = 1$ for every edge~$e$, while
\[[v_{bl}:e_b] = -1 \qquad \text{and} \qquad [v_{br}:e_b] = 1\]
for the bottom edge~$e_b$.

For each face $F$ we define the ring $A_F$ to be the monoid
$R$\nbd-algebra
\[A_F=R[T_F \cap \bZ^2]\] where $T_F$ is the tangent cone $T_F =
\mathrm{span}_{\geq 0}\{s-f \mid s \in S,\,f\in F\}$ of~$S$ at~$F$.
For example, the ring $A_S$ is the \textsc{Laurent} polynomial ring in
two indeterminates $R[x,\,x\inv,\,y,\,y\inv]$, while $A_{v_{br}} =
R[x\inv,\, y]$ is a polynomial ring. In general we have $A_F \subseteq
A_G$ if $F \subseteq G$.

For a face $F$ of~$S$ we let $v_F \in \bZ^2$ denote its barycentre, or
centre of mass. For a pair of faces $F\subset G$, the difference
$v_G-v_F$ represents a monomial~$m_{FG}$ in the centre of~$A_F$; for
example, $m_{v_{bl}e_b} = x$, $m_{v_{tl}S} = xy\inv$ and $m_{e_rS} =
x\inv$. The algebra $A_G$~is the localisation of~$A_F$ by~$m_{FG}$:
\begin{equation}
  \label{eq:localisation}
  \parbox{0.8 \textwidth}{\it
    For every element \(u \in A_G\) there is \(k_0 \geq 0\) such
    that\break \(m_{FG}^k u \in A_F\) for all \(k \geq k_0\).}
\end{equation}
We also choose $m_{\emptyset F}$ to be the monomial represented by the
lattice point $v_F$, for every face~$F$ of~$S$, so that $m_{\emptyset
  e_l} = x\inv$ and $m_{\emptyset v_{tl}} = x\inv y$. Clearly,
\begin{equation}
  \label{eq:monoids_compatible}
  \parbox{0.8 \textwidth}{\it \(m_{FH} = m_{FG} \cdot m_{GH}\) for all
    triples of (possibly empty) faces \(F \subseteq G \subseteq H\) of
    \(S\).}
\end{equation}

When $F$ is a codimension\nbd-$1$ face of~$G$, the monomial $m_{FG}$
is the indeterminate $z\in\{x,\, x\inv,\, y,\, y\inv\}$ such that
$\{z,\,z\inv\}\subset A_G$ and $z\inv\notin A_F$.

We denote by $\fp$ the join semi-lattice of non-empty faces of~$S$
ordered by inclusion, and by $\Chd{R}$ the category of cochain
complexes of right $R$-modules.

\begin{definition}\label{def:presheaf}
  \begin{enumerate}
  \item An \emph{$\fp$-diagram~$X$ of cochain complexes\/} is a
    functor $X \colon {\fp} \rTo \Chd{R}, \ F \mapsto X_F$ equipped
    with extra data giving each $X_F$ the structure of a cochain
    complex of (right) $A_F$-modules, such that for each pair of
    non-empty faces $F \subseteq G$ of~$S$ the corresponding structure
    map $s_{FG}\colon X_F \rTo X_G$ is in fact $A_F$\nbd-linear.
  \item The $\fp$\nbd-diagram~$X$ is called \emph{bounded} if the
    cochain complex $X_F$ is bounded for every face~$F$ of~$S$.
  \item An $\fp$\nbd-diagram~$X$ is said to be {\em quasi-coherent} if
    for each pair of non-empty faces $F\subset G$ of~$S$ the adjoint
    map \[X_F\otimes_{A_F}A_G \rTo X_G\] of~$s_{FG}$ is an isomorphism
    of $A_G$-module cochain complexes.
  \end{enumerate}
\end{definition}

As a matter of notation, given a functor $X \colon {\fp} \rTo \Chd{R}$
we denote the differential of the cochain complex~$X_F$ by~$d_F$.

An example of a quasi-coherent $\fp$\nbd-diagram (concentrated in a
single cochain degree) is the diagram ${D}(k)$ shown in
Fig.~\ref{fig:dk}. We have $D(k)_F=A_F$, and the structure maps
$s_{FG}$ are the inclusions followed by multiplication with the
monomial~$m_{FG}^{-k}$.

\begin{figure}
\begin{diagram}[PS,w=0.18\textwidth,h=0.08\textwidth,tight,midshaft]
 R[x,y\inv]   & \rTo^{x^{-k}} & R[x,x\inv,y\inv]   & \lTo^{x^k} &  R[x\inv,y\inv]       \\
 \dTo^{y^k}    &               &      \dTo^{y^k}      &            &       \dTo^{y^k}        \\
R[x,y,y\inv]  & \rTo^{x^{-k}}  & R[x,x\inv,y,y\inv] & \lTo^{x^k} &  R[x\inv,y,y\inv]     \\
  \uTo^{y^{-k}}&               &      \uTo^{y^{-k}}   &            &       \uTo^{y^{-k}}     \\
   R[x,y]      & \rTo^{x^{-k}}  & R[x,x\inv,y]        & \lTo^{x^k} &  R[x\inv,y] 
\end{diagram}
\caption{The diagram ${D}(k)$}
\label{fig:dk}
\end{figure}

\begin{construction}\label{con:extend}
  Let $C$ be a bounded cochain complex consisting of finitely
  generated free $R[x,\,x\inv,\,y,\,y\inv]$-modules, with a
  specified finite basis $B_a \subset C^a$ so that we have an
  identification $C_a = \bigoplus_{B_a} R[x,\,x\inv,\,y,\,y\inv]$.

  We will show how to construct a sequence $k_i$ of non-negative
  integers and a bounded $\fp$\nbd-diagram~$Y$ with $Y_S = C$ such
  that $Y^j \iso \bigoplus_{B_j} D(k_j)$ for every $j \in \bZ$.

  \medbreak

  We start by making the default convention that all numbers~$k_j$
  which are not explicitly defined are taken to be~$0$, and similarly
  that all diagrams~$Y^j$ and all boundary maps that are not
  explicitly defined are taken to be trivial.

  \medbreak

  If $C$~is the $0$\nbd-complex we can choose the trivial diagram with
  all values~$0$. Otherwise, $C^i$ is trivial for $i<a$, say, and $C^a
  \neq 0$. For a face~$F$ of~$S$ we define $Y^a_F = A_F[B_a]$ to be
  the free $A_F$\nbd-module on the basis~$B_a$, and for a pair of
  faces $F \subseteq G$ we define the structure map $s^a_{FG} \colon
  Y^a_F \rTo Y^a_G$ as the map induced by the inclusion $A_F \subseteq
  A_G$ and the identity on basis elements. In other words, we define
  $Y^a = \bigoplus_{B_a} D(0)$. We set $k_a = 0$.

  If $C^a$ is the only non-trivial module in~$C$ the construction
  process terminates. Otherwise, we define $Y^{a+1}$ by setting
  $Y^{a+1}_F = A_F[B_{a+1}]$, the free $A_F$\nbd-mod\-ule with basis
  $B_{a+1}$. The structure maps will be of the form ``inclusion
  followed by multiplication with the monomial $m_{FG}^{k_{a+1}}$'',
  for some suitable~$k_{a+1}$ to be determined presently. In other
  words, we will have $Y^{a+1} = \bigoplus_{B_{a+1}} D(k_{a+1})$.

  To find a suitable value of~$k_{a+1}$ recall that we also need to
  construct differentials $d^{a+1}_F\colon Y_F^a\to Y_F^{a+1}$, for
  all the faces $F$ of~$S$, which are compatible with the structure
  maps. We choose $d^a_S = d^a_C$ to be the given differential of the
  cochain complex~$C$. For a proper face~$F$ and fixed basis element
  $b \in B_a$ the image~$w_{F,b}$ of~$b$ under the composition
  \[\bigoplus_{B_a} A_F = Y_F^a \rTo[l>=3em]^{s^a_{FS}}
  Y_S^a = C^a \rTo[l>=3em]^{d_C^a} C^{a+1} = Y_S^{a+1} =
  \bigoplus_{B_{a+1}} A_S\] is such that $m_{FS}^k w_{F,b} \in
  \bigoplus_{B_{a+1}} A_F = Y_F^{a+1}$ for large~$k$,
  by~\eqref{eq:localisation}. We choose a number $k_{a+1}= k$ large
  enough to work for all $b \in B_a$, and all proper faces~$F$ of~$S$.

  Now we can define our $A_F$\nbd-linear differential $d_F^a$ uniquely
  by the requirement that it sends $b \in B_a$ to
  $m_{FS}^{k_{a+1}}w_{F,b} \in Y_F^{a+1}$ for a proper face~$F$
  of~$C$.

  We claim that this yields a map of $\fp$\nbd-diagrams $d^{a-1}
  \colon Y^{a-1} \rTo Y^a$. We need to verify $s^{a+1}_{FG} \circ
  d^a_F = d^a_G \circ s^a_{FG}$ for faces $F \subset G$. In fact, this
  equality holds for $G = S$ by construction: the basis element $b \in
  B_a$ is sent to $w_{F,b}$ by the composition on the right, while on
  the left we have $s^{a+1}_{FS} \circ d^a_F (b) = s^{a+1}_{FS}
  (m_{FS}^{k_{a+1}} w_{F,b}) = m_{FS}^{-k_{a+1}} m_{FS}^{k_{a+1}}
  w_{F,b} = w_{F,b}$ as required. If $G$~is a proper face of~$S$ we
  note that, by construction and the argument just given, we have
  \[s_{GS}^{a+1} \circ s^{a+1}_{FG} \circ d^a_F = s_{FS}^{a+1} \circ
  d^a_F = d_S^a \circ s_{FS}^a = d_S^a \circ s_{GS}^a \circ s_{FG}^a =
  s_{GS}^{a+1} \circ d^a_G \circ s^a_{FG} \ ;\] since
  $s_{GS}^{a+1}$~is injective, the claim follows.

  If there are no further non-trivial entries in~$C$ the construction
  terminates. Otherwise, we extend from $Y^j$ to~$Y^{j+1}$ by
  repeating the process above: choose a sufficiently large integer
  $k_{j+1} \geq 0$ so that, for each proper face~$F$ of~$S$ and each
  basis element $b \in B_j$, we have $m_{FS}^{k_{j+1}} w_{F,b} \in
  \bigoplus_{B_{j+1}} A_F = Y_F^{j+1}$, where $w_{F,b}$ is the image
  of~$b$ under the composition
  \[\bigoplus_{B_j} A_F = Y_F^j \rTo[l>=3em]^{s_{FS}^j}
  Y_S^j = C^j \rTo[l>=3em]^{d_C^j} C^{j+1} = Y_S^{j+1} =
  \bigoplus_{B_{j+1}} A_S \ .\] We let $Y^{j+1} = \bigoplus_{B_{j+1}}
  D(k_{j+1})$, and define differentials $d_F^j$ by the requirement
  that they send the basis element $b \in B_j$ to $m_{FS}^{k_{j+1}}
    w_{F,b}$.

  \medbreak

  Since $C$ is bounded this process terminates, and results in a
  quasi-coherent $\fp$\nbd-diagram~$Y$ with $Y_S = C$. It might be
  worth pointing out that the composition of two differentials is the
  zero map as required. Indeed, for every~$i$ the structure map
  $s_{FS}^{i+2}$~is injective, and we have
  \[s_{FS}^{i+2} \circ d_F^{i+1} \circ d_F^i = d_C^{i+1} \circ d_C^i
  \circ s_{FS}^i = 0\] (since $d_C$~is a a differential) so that
  $d_F^{i+1} \circ d_F^i = 0$.
\end{construction}

\section{Double complexes from diagrams. \v Cech complexes}
\label{sec:cech-complexes}

Let $P$ be a poset. We suppose that $P$ is equipped with a strictly
increasing {\it degree function\/} $\deg \colon P \rTo \bN$, and an
{\it incidence function\/}
\[[\nix : \nix] \colon P \times P \rTo \bZ\]
satisfying the following properties:
\begin{itemize}
\item[(DI1)] $[x:y] = 0$ unless $x<y$ and $\deg(y) = 1+\deg(x)$;
\item[(DI2)] for all $x < z$ with $\deg(z) = 2+\deg(x)$, the open
  interval $I(x:z) = \{y \in P \,|\, x < y < z\}$ is finite, and we
  have
  \[\sum_{y \in I(x:z)} [x:y] \cdot [y:z] = 0 \ ;\]
\item[(DI3)] for $z \in P$ with $\deg(z) = 1$ the set $I(<z) = \{y \in P
  \,|\, y < z\}$ is finite, and we have
  \[\sum_{y \in I(<z)} [y:z] = 0 \ .\]
\end{itemize}

Now let $X \colon P \rTo R\text{-}\mathrm{Mod},\ x \mapsto X_x$ be a
diagram of $R$\nbd-modules. We define the {\it \textsc{\v Cech}
  complex of~$X$\/}, denoted $\vgamma(X)$, to be the following cochain
complex with
\[\big( \vgamma(X) \big)^n = \bigoplus_{\substack{x \in P \\ \deg(x) =
    n}} X_x\] and differential induced by the structure maps of~$X$
modified by incidence numbers~$[x:y]$. It follows from~(DI2) that this
defines indeed a cochain complex.

\medbreak

More generally, for a $P$\nbd-indexed diagram~$X$ of cochain complexes
of $R$\nbd-modules we define a double complex~$D^{\ast,\ast}$, and
define the {\it \textsc{\v Cech} complex of~$X$\/}, denoted
$\vgamma(X)$, to be the totalisation:
\[\vgamma(X) = \totds D^{\ast,\ast}\]
The double complex~$D^{\ast,\ast}$ is defined by saying that
\begin{itemize}
\item the $q$th row~$D^{\ast,q}$ is the \textsc{\v Cech} complex
  $\vgamma(X^q)$ of~$X^q$, the $P$\nbd-indexed diagram of
  $R$\nbd-modules consisting of all the cochain level~$q$ terms of
  entries in~$X$, and
\item the vertical differential in column~$p$ is induced by the
  differentials of the cochain complexes $X_x$ with $\deg(x) = p$,
  modified by the sign~$(-1)^p$. 
\end{itemize}
By construction, vertical and horizontal differentials anti-commute.

\medbreak

We will apply this construction in specific cases only. For example,
for $P = \fp$ we can choose $\deg$ to be the dimension function, and
use the incidence numbers introduced at the beginning
of~\S\ref{sec:diagrams}. This data satisfies the conditions (DI1--3)
above. In particular, every $\fp$\nbd-diagram~$X$ in the sense of
definition~\ref{def:presheaf} gives rise to a double complex
$D^{\ast,\ast}$ of $R$\nbd-modules, and a cochain complex $\vgamma(X)
= \totds D^{\ast,\ast}$ of $R$\nbd-modules. Explicitly, we have
\begin{equation}
  \label{eq:dblgammay}
  D^{i,j} := \bigoplus_{\dim\,F=i} X^j_F
\end{equation}
with horizontal and vertical differentials induced by
\[d_h=[F:G]s_{FG}\ , \quad d_v=(-1)^i \cdot \bigoplus_{\dim\,F=i}d_F\]
where $[F:G]$ are our chosen incidence numbers, and
\[\vgamma(X)^n=\big(\totds D^{\ast,\ast}\big)^n = \bigoplus_{i+j=n}
D^{i,j}\] equipped with differential $d=d_h+d_v$.

\section{\v{C}ech cohomology. Computations}

Let $C$ and~$Y$ be as in Construction~\ref{con:extend}, and let
$D^{\ast,\ast}$ be the double complex (\ref{eq:dblgammay}) used to
define $\vgamma(Y)$. In particular the module $C^t = D^{2,t}$ is a
free $R[x,\,x\inv,\,y,\,y\inv]$-module with basis~$B_t$. By
construction, $D^{\ast,t}$ is the \textsc{\v{C}ech} complex of a
$B_t$-indexed direct sum of diagrams $D(k_t)$, where $k_t \geq 0$
depends on~$t$ according to Construction~\ref{con:extend}. (As before
we use the symbol $D(k)$ to denote the diagram of
Fig.~\ref{fig:dk}. We will never refer to an object~`$D$' alone, so
our notation should not lead to any confusion.) To compute the
cohomology of $D^{\ast,t}$ it is thus enough to compute the cohomology
of $\vgamma\big(D(k)\big)$ for $k\geq0$.

Recall that $S=[-1,1]^2$. We denote by $kS$ the {\em $k$-th dilate of
  $S$} which is $[-k,k]^2$. By $R[kS\cap \bZ^2]$ we mean the free
$R$-module with basis $kS \cap \bZ^2 = \{(i,j)\in \bZ^2 \mid -k\leq
i,j \leq k\}$, the set of lattice points in~$kS$.

\medbreak

We now want to consider the following cochain complex of
$R$\nbd-modules, an augmented version of the complex $\vgamma(Y)$:

\begin{equation}
  \label{eq:complex}
  \begin{split}
    0\rTo R[kS\cap\bZ^2]\rTo^{d\inv} \hskip 0.45 \textwidth \\ \qquad
    \vgamma\big(D(k)\big)^0 \rTo^{d^{0}} \vgamma\big(D(k)\big)^1
    \rTo^{d^{1}} \vgamma\big(D(k)\big)^2 \rTo 0
  \end{split}
\end{equation}
The map $d\inv$ is given by multiplication with the monomials
$m_{\emptyset v}^{-k}$ for the various vertices~$v$ of~$S$, followed
by an inclusion map. More explicitly, bar the zero terms this is the
complex
\begin{align*}
  R[kS\cap \bZ^2] \rTo^{d\inv}
  \begin{matrix}
    R[x,y] \\ \oplus \\ R[x\inv,y] \\ \oplus \\  R[x,y\inv]  \\ \oplus \\ R[x\inv,y\inv] 
  \end{matrix}
  \rTo^{d^0}
  \begin{matrix}
    R[x,x\inv,y] \\ \oplus \\ R[x,y,y\inv] \\ \oplus \\  R[x\inv,y,y\inv]  \\ \oplus \\ R[x,x\inv,y\inv] 
  \end{matrix}
  \rTo^{d^1} R[x,\, x\inv,\, y,\, y\inv]
\end{align*}
where $d\inv$, $d^0$ and $d^1$ are given by matrices
\begin{multline*}
  d\inv=
  \begin{pmatrix}
    x^{k}y^{k} \\
    x^{-k}y^{k} \\
    x^{k}y^{-k} \\
    x^{-k}y^{-k} 
  \end{pmatrix}
  \ , \quad
  d^0=
  \begin{pmatrix}
    -x^{-k} &   x^k  &   0    &   0   \\
    y^{-k} &    0   &  -y^k   &   0   \\
    0    & -y^{-k} &   0    &  y^k   \\
    0    &    0   & x^{-k} &  -x^k  
  \end{pmatrix}
  \\ \noalign{\smallskip} \text{and }
  d^1=\big( y^{-k}, x^{-k}, x^k, y^k \big) \ ;
\end{multline*}
the entries of these matrices are all of the form $[F:G] \cdot
m_{FG}^{-k}$, for (possibly empty) faces $F \subset G$.

\begin{lemma}
  \label{lem:horzhom}
  The complex~{\rm \eqref{eq:complex}} is exact for $k \geq 0$.
\end{lemma}

\begin{proof}
  It is easy to see that $d^1$ is surjective; for example, we can
  write any element $p \in R[x,\, x\inv,\, y,\, y\inv]$ as a sum $p =
  p_+ + p_-$ where $p_+ \in R[x,\, x\inv][y]$ and $p_- \in y\inv
  R[x,\, x\inv][y\inv]$, and note that $d^1$~maps the quadruple
  $(y^kp_+,\, 0,\, 0,\, y^{-k}p_-)$ to~$p$.

  \smallbreak

  To show that the complex is exact at $\vgamma\big(D(k)\big)^1$ for a
  given boundary $e_1\in\mathrm{ker}(d^1)$ we will construct an
  explicit cochain $e_0\in\vgamma\big(D(k)\big)^0$ for which $d^0(e_0)
  = e_1$.

  Fix exponents $i$ and~$j$. On the one hand, the coefficient of the
  monomial~$x^i y^j$ in $d^1(e_1) = 0$ is zero. On the other hand, we
  can systematically work out which terms of the components of~$e_1$
  contribute to the coefficient of~$x^i y^j$ in $d^1(e_1)$, as
  follows.
  \begin{enumerate}[C{a}se~1:]
  \item $|i|>k$ and $|j|>k$. The coefficient of $x^iy^j$ receives
    contribution from exactly two of the four components of~$e_1$.
    By symmetry we may assume that $i, j>k$ (the other cases being
    similar); then the first and second components of~$e_1$ must
    contain terms of the form $ax^{i}y^{j+k}$ and~$bx^{i+k}y^j$,
    respectively, with $a+b = 0$. We set $z_{i,j} =
    (-ax^{i+k}y^{j+k},\,0,\,0,\,0)$ and note that $d^0 (z_{i,j}) =
    (ax^{i}y^{j+k},\,bx^{i+k}y^j,\,0,\,0)$.
  \item $|i|\leq k$ and $|j|>k$. The coefficient of~$x^iy^j$ receives
    contributions from three of the components of~$e_1$. Let us again
    assume that both $i$ and~$j$ are non-negative. Then the first
    three components of~$e_1$ must contain terms of the form
    $ax^{i}y^{j+k}$, $bx^{i+k}y^j$ and~$cx^{i-k}y^j$, respectively,
    with $a+b+c=0$. We set $z_{i,j} = (bx^{i+k}y^{j+k},\,
    -cx^{i-k}y^{j+k},\, 0,\, 0)$ and note that $d^0 (z_{i,j}) =
    (ax^{i}y^{j+k},\, bx^{i+k}y^j,\, cx^{i-k}y^j,\, 0)$.
  \item $|i|>k$ and $|j| \leq k$. This is dealt with in a manner
    similar to the previous case.
  \item $|i| \leq k$ and $|j| \leq k$. The coefficient of~$x^iy^j$
    receives contributions from all four components of~$e_1$, which
    must contain terms of the form $ax^{i}y^{j+k}$, $bx^{i+k}y^j$,
    $cx^{i-k}y^j$ and~$dx^{i}y^{j-k}$, respectively, with
    $a+b+c+d=0$. We choose
    \[z_{i,j} = \big( (b+d)x^{i+k}y^{j+k},\, -c x^{i-k} y^{j+k},\, d
    x^{i+k}y^{j-k},\, 0 \big)\] and note that $d^0 (z_{i,j}) =
    (ax^{i}y^{j+k},\, bx^{i+k}y^j,\, cx^{i-k}y^j,\,
    dx^{i}y^{j-k})$. (Unlike before, the definition of~$z_{i,j}$ does
    involve a choice between many alternatives. The one given will
    do.)
  \end{enumerate}
  Now define $e_0 = \sum_{i,j \in \bZ} z_{i,j}$; this is a finite sum
  as only finitely many of the~$z_{i,j}$ can be non-zero. By
  construction we have $d^0(e_0) = e_1$ so that
  $\mathrm{im}(d^0)=\ker(d^1)$ as required.

  \smallbreak

  We now show that $\ker (d^0)$ coincides with the image of~$d\inv$.
  An element~$e_0$ of $\vgamma \big( D(k) \big)^0$ is of the form
  \[e_0 = \Big( \sum_{i,j \geq 0} a_{ij}x^iy^j,\, \sum_{i \leq 0 \leq
    j} b_{ij} x^iy^j,\, \sum_{j \leq 0 \leq i} c_{ij} x^i y^j,\,
  \sum_{i,j \leq 0} d_{ij} x^i y^j \Big) \ ,\] with all sums being
  finite. If $d^0 (e_0) = 0$, that is, if $e_0 \in \ker d^0$, then in
  particular (by considering the first component)
  \[-x^{-k} \cdot \sum_{i,j \geq 0} a_{ij}x^iy^j + x^k \cdot \sum_{i
    \leq 0 \leq j} b_{ij} x^iy^j = 0 \ \in R[x,x\inv,y] \ .\]
  This implies for all $j \geq 0$ that
  \begin{align*}
    a_{ij} &= b_{i-2k,j} && \text{for } 0 \leq i \leq 2k \ ,\\
    a_{ij} &= 0 && \text{for } i > 2k \ ,\\
    b_{ij} &= 0 && \text{for } i < -2k \ .
  \end{align*}
  Considering the other components of~$d^0(e_0)$ in a similar manner
  shows that all coefficients with subscripts $|i|>2k$ or $|j|>2k$
  vanish, and that
  \begin{align*}
    a_{ij} &= c_{i,j-2k} && \text{for } 0 \leq j \leq 2k,\ i \geq 0 \ ,\\
    b_{ij} &= d_{i,j-2k} && \text{for } 0 \leq j \leq 2k,\ i \leq 0 \ ,\\
    c_{ij} &= d_{i,j-2k} && \text{for } 0 \leq j \leq 2k,\ j \leq 0 \ .
  \end{align*}
  Now we choose
  \[e_{-1} = \sum_{-k \leq i,j \leq k} a_{i+k,j+k} x^i y^j \ \in R[kS
  \cap \bZ^2] \ ,\] and our characterisation of the coefficients
  of~$e_0$ above immediately gives that $d^{-1} (e_{-1}) = e_0$.

  \smallbreak

  Finally, it remains to observe that $d\inv$ is injective; in fact,
  each of the four components of~$d\inv$ is injective by itself.
\end{proof}

\begin{corollary}
  \label{cor:H_of_D_t}
  The complex $D^{\ast,t}$ is quasi-isomorphic to the finitely
  generated free $R$-module (considered as a cochain complex
  concentrated in degree~$0$)
  \[\bigoplus_{B_t} R[k_tS \cap \bZ^2]\] via the map~$d\inv$
  from the preceding lemma.\qed
\end{corollary}

\begin{proposition}
  \label{prop:bdfgfree}
  Let $C$ and~$Y$ be as in Construction~\ref{con:extend}. There is a
  quasi-isomorphism $\chi \colon \bfree \rTo \vgamma(Y)$ from a
  bounded cochain complex~$\bfree$ of finitely generated free
  $R$-modules with ${\bfree}^{\hskip 0.6pt i} \cong \bigoplus_{B_i}
  R[k_iS\cap\bZ^2]$.
\end{proposition}

\begin{proof}
  Let $D^{\ast,\ast}$ denote the right half-plane double
  complex~(\ref{eq:dblgammay}) associated to~$Y$ as at the beginning
  of this subsection. We define a new double complex~$E^{\ast,\ast}$
  which agrees with~$D^{\ast,\ast}$ everywhere except that in
  column~$-1$ we put the free $R$\nbd-modules $E^{-1,t} =
  \bigoplus_{B_t} R[k_tS \cap \bZ^2]$, with vertical differential
  induced by the negative of the differential of~$Y$ and horizontal
  differential~$d\inv$ as above. The resulting double complex has
  exact rows, by Corollary~\ref{cor:H_of_D_t}, hence the induced map
  $\chi \colon E^{-1,\ast} \rTo \vgamma (Y) = \totds D^{\ast,\ast}$ is
  a quasi-isomorphism by Lemma~\ref{lem:augmented_double}. (To apply
  the Lemma we may need to re-index $C$ and~$Y$ temporarily to make
  sure that all non-zero entries live in non-negative cochain
  degrees.)
\end{proof}

\section{The nerve of the square. More Novikov rings}
\label{sec:nerve}

In this section we will describe diagrams involving both the
\textsc{Novikov} rings from Part \ref{part:1} along with other rings
we will describe shortly. For each face $F$ we will construct a
diagram $\bsd_F$, the {\textsc{\v{C}ech}} complex of which will be
quasi-isomorphic to $A_F$ considered as a cochain complex concentrated
in degree zero.

\begin{definition}
  For a non-empty face $F$ of the square $S$ we make the following
  definitions.
  \begin{enumerate}
  \item The {\em star of $F$}, denoted $\mathrm{st}(F)$, is the set of
    faces of $S$ which contain~$F$.
  \item We will denote by $\mathcal{N}_F$ the {\em nerve of
      $\mathrm{st}(F)$}, the simplicial complex on the vertex set
    $\mathrm{st}(F)$ in which the faces are given by the sequences of
    strict inclusions of elements of $\mathrm{st}(F)$. The faces of
    $\mathcal{N}_F$ will also referred to as {\em flags} of faces of
    $S$.
  \end{enumerate}
\end{definition}

For each face $\tau$ of~$\mathcal{N}_F$ we will define a ring
$\comp{\tau}$. Note that $\tau$ can occur in $\mathcal{N}_F$ for many
faces $F$; the definition of $\comp{\tau}$ does not depend on $F$,
however, which is reflected by the notation.

\begin{definition}
  By our previous definitions we have $A_{v_{bl}}=R[x,y]$,
  $A_{e_b}=R[x,x\inv,y]$ and $A_S=\LL$. We now define the following
  rings:
  \begin{align*}
    &\comp{v_{bl}}=R\powers{x,y}           &&   \comp{v_{bl},S}=R\nov{x,y}      \\
    &\comp{e_b}=R[x,x\inv]\powers{y}      &&   \comp{e_b,S}=R[x,x\inv]\nov{y}  \\
    &\comp{v_{bl},e_b}=R\nov{x}\powers{y}  &&   \comp{v_{bl},e_b,S}=R\nov{x}\nov{y}
  \end{align*}
  Also, we define $\comp{S}=A_S$. If $e_b$ is replaced by $e_l$
  throughout, we have the same definitions with $x$ and~$y$
  swapped. To replace the subscript~``$l$'' by~``$r$'' everywhere we
  replace $x$ by~$x\inv$, and finally to replace the subscript ``$b$''
  by ``$t$'' everywhere we replace $y$ by~$y\inv$.
\end{definition}

This defines a ring $\comp{\tau}$ for any non-empty flag~$\tau$ of
faces of~$S$. For example, we have
\begin{align*}
  \comp{v_{br}} &=  R\powers{x\inv,y} \ \text{and} \\
  \comp{v_{tr},\, e_r} & = R\nov{y\inv}\powers{x\inv}\ .  
\end{align*}

\medbreak

For each face $F$ of the square we have a diagram of $R$\nbd-modules
\[\bsd_F\colon\mathcal{N}_F\rTo R\text{-}\mathrm{Mod}\,, \quad \tau
\mapsto \comp{\tau}\] with maps given by inclusion of rings. Here and
elsewhere we consider~$\mathcal{N}_F$ as a poset ordered by inclusion
of flags.

We first exhibit the case when $F = \{v\}$ is a vertex. The star
of~$v$ is the set $\{v,e_x,e_y,S\}$ where $e_x$ and $e_y$ are the two
edges incident to $v$.  Specifically we take $v = v_{tr} = (1,1)$,
$e_x = e_r$ to be the right vertical edge of~$S$, and $e_y = e_t$ to
be the top horizontal edge (the situation for the other vertices will
be similar). The diagram $\bsd_v$ is given in
Fig.~\ref{fig:nervevert}. The entries are specified using both the
abstract notation $\comp{\tau}$ as well as the concrete
\textsc{Novikov} rings.
\begin{figure}[bt]
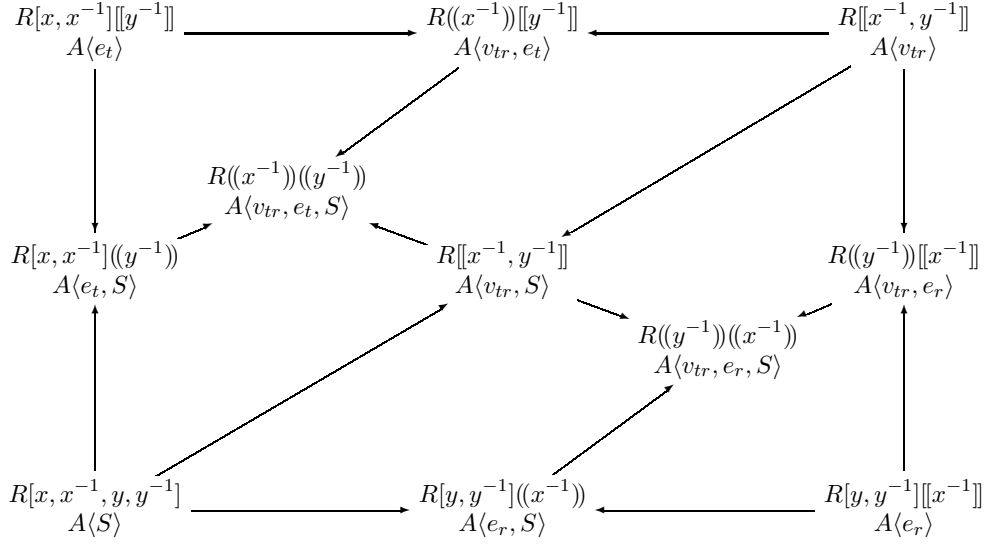

\resizebox{\textwidth}{!}{%
\begin{diagram}[PS,h=1.6em,w=0.5em]
\ou{\fy}{e_t}  &&&          &           &\rTo     &\ou{\fvy}{v_{tr},e_t}&  \lTo  &            &&&          & \ou{\fv}{v_{tr}} \\
\dTo           &&&          &           &\ldTo(2,4)&             &          &            &&&\ldTo(6,6)& \dTo             \\
               &&&          &           &          &             &          &            &&&          &                  \\
               &&&          &           &          &             &          &            &&&          &                  \\
               &&&      &\ou{\fvys}{v_{tr},e_t,S}&      &             &          &            &&&          &                  \\
               &&&\ruTo(4,2)&           &\luTo(2,2)&             &          &            &&&          &                  \\
\ou{\fys}{e_t,S}&&&         &           &          &\ou{\fvs}{v_{tr},S}&         &            &&&          & \ou{\fvx}{v_{tr},e_r} \\
\uTo           &&&          &           &\ruTo(6,6)&             &\rdTo(2,2)&            &&&\ldTo(4,2)& \uTo             \\
               &&&          &           &          &             &       &\ou{\fvxs}{v_{tr},e_r,S}&&&      &                  \\
               &&&          &           &          &             &\ruTo(2,4)&            &&&          &                  \\
               &&&          &           &          &             &          &            &&&          &                  \\
               &&&          &           &          &             &          &            &&&          &                  \\
\ou{\fs}{S}    &&&          &           &\rTo     &\ou{\fxs}{e_r,S}& \lTo   &            &&&          &  \ou{\fx}{e_r}  
\end{diagram}}
\caption{The diagram $\bsd_v$ for $v= v_{tr} = (1,1)$}
\label{fig:nervevert}
\end{figure}

For the case when $F$ is an edge we take $F=e_x$ as defined above (the
other cases being similar), and note that the diagram $\bsd_{e_x}$
appears as a restriction of $\bsd_{v}$ since $\mathcal{N}_{e_x}$ is an
order ideal of $\mathcal{N}_{v}$.  Explicitly, $\bsd_{e_r}$ looks like
this:
\begin{equation}
  \begin{matrix}
    \fs & \rTo & \fxs & \lTo & \fx \\
    \noalign{\vskip 1.5 \smallskipamount}
    \comp{S} && \comp{e_r,S} && \comp{e_r}
  \end{matrix}
  \label{diag:Nex}
\end{equation}

Finally, if $F$ is the square~$S$ itself, then the nerve
$\mathcal{N}_S$ of $\mathrm{st} (S)$ is just~$\{S\}$ and the
diagram~$\bsd_{S}$ consists of only~$A_S = \comp{S}$.

\medbreak

Back to general~$F$, we equip $\mathcal{N}_F$ with degree and
incidence functions in the sense of \S\ref{sec:cech-complexes}: the
degree function is given by the (simplicial) dimension, which assigns
to a flag with $k+1$ entries dimension~$k$, and the standard
simplicial incidence numbers. To explain the latter, note that a
flag~$\tau$ is totally ordered, so we can let $d_i (\tau)$ denote the
flag obtained by omitting the $i$th entry ($0 \leq i \leq \dim \tau$)
and set $[d_i(\tau):\tau] = (-1)^i$; all other incidence numbers
vanish.  --- The diagram~$\bsd_F$ then has an associated \textsc{\v
  Cech} complex $\vgamma (\bsd_F)$. Explicitly,
\[\vgamma(\bsd_F)^t = \bigoplus_{\substack{\tau \in \mathcal{N}_F \\
    \dim\,\tau=t}} \comp{\tau}\] with differential induced by
\[\comp{\tau}\rTo^{[\tau:\mu]}\comp{\mu} \ .\]

\goodbreak

\section{Decomposing diagrams. \v Cech cohomology calculations}

We keep the notation from~\S\ref{sec:nerve}. For a fixed face~$F$
of~$S$, the $R$\nbd-module diagram~$\bsd_F$ is in fact a diagram of
$A_F$\nbd-modules: all its entries contain $A_F$ as a subring, and all
structure maps are $A_F$\nbd-linear. In particular, its \textsc{\v
  Cech} complex $\vgamma(\bsd_F)$ is an $A_F$\nbd-module
complex. Moreover, by the property~(DI3) of incidence numbers, the
inclusion maps of subrings assemble to a map of $A_F$\nbd-module
cochain complexes
\begin{equation}
  \label{eq:sigma_F}
  \sigma_F \colon A_F \rTo \vgamma (\bsd_F) \ ,
\end{equation}
where we consider
the left hand side as a cochain complex concentrated in degree~$0$ as
usual.

\begin{lemma}
  \label{lem:morse}
  The map~$\sigma_F$ is a quasi-isomorphism.
\end{lemma}

\begin{proof}
  We note that there is nothing to prove in case $F = S$ as
  $\vgamma(\bsd_S) = \comp{S}=A_S$ and $\sigma_S = \id_{A_S}$.

  \smallskip

  Next suppose that $F$ is an edge of~$S$. We will treat the case $F =
  e_r = \{1\} \times [-1,1]$ only, the other cases are identical apart
  from possible coordinate changes replacing an indeterminate $x$
  or~$y$ by its inverse, or swapping their roles. --- The diagram
  $\bsd_{e_r}$ is depicted in~(\ref{diag:Nex}). Considered as a
  diagram of $R$\nbd-modules it decomposes as the direct sum of two
  diagrams of the same shape, corresponding to non-positive and
  positive powers of~$x$ respectively: $\bsd_{e_r} = \mathcal{D}_+
  \oplus \mathcal{D}_-$ where
  \begin{align*}
    \mathcal{D}_- &= \big( \ \hphantom{x}R[x\inv,y,y\inv] \rTo
    \hphantom{x}R[y,y\inv]\powers{x\inv} \lTo R[y,y\inv]\powers{x\inv} \ \big) \\
    \noalign{\noindent and} \mathcal{D}_+ &= \big( \
    xR[x,y,y\inv]\hphantom{{}\inv} \rTo xR[x,y,y\inv]
    \hphantom{{}\inv[]} \lTo 0 \ \big) \ .
  \end{align*}
  Upon application of~$\vgamma$, the short exact sequence of diagrams
  \[0 \rTo \mathcal{D}_+ \rTo \bsd_{e_r} \rTo \mathcal{D}_- \rTo 0\]
  translates into a short exact sequence of \textsc{\v Cech} complexes
  \[0 \rTo \vgamma (\mathcal{D}_+) \rTo \vgamma (\bsd_{e_r})
  \rTo^\beta \vgamma (\mathcal{D}_-) \rTo 0 \ .\] Now
  $\vgamma(\mathcal{D}_+)$ is acyclic (it is a two-step complex with
  only non-trivial differential being an isomorphism) so that
  $\beta$~is a quasi-isomorphism. Introducing the diagrams
  \begin{align*}
    \mathcal{E}_0 &= \big( \ 0 \rTo R[y,y\inv]\powers{x\inv} \lTo
    R[y,y\inv]\powers{x\inv} \ \big) \\
    \noalign{\noindent and}
    \mathcal{E}_1 &= \big( \ R[x\inv,y,y\inv] \rTo 0 \lTo 0 \ \big)
  \end{align*}
  we have a short exact sequence
  \[0 \rTo \mathcal{E}_0 \rTo \mathcal{D}_- \rTo \mathcal{E}_1
  \rTo 0\] and consequently a short exact sequence
  \[0 \rTo \vgamma (\mathcal{E}_0) \rTo \vgamma (\mathcal{D}_-)
  \rTo^\gamma \vgamma (\mathcal{E}_1) \rTo 0 \ .\] Now $\vgamma
  (\mathcal{E}_0)$ is acyclic by the same reasoning as before so that
  $\gamma$~is a quasi-isomorphism to $\vgamma (\mathcal{E}_1) =
  A_{e_r}$. As clearly $\gamma \circ \beta \circ \sigma_{e_r} =
  \id$ 
  it follows that $\sigma_{e_r}$~is a quasi-isomorphism.

  \smallskip

  We finally deal with the case $F = v$ a vertex. We will treat
  $v=v_{tr} = (1,1)$ explicitly, the other cases are similar (and
  follow formally by change of variables). --- Informally speaking, we
  decompose the diagram $\bsd_{v_{tr}}$ as a direct sum of $R$\nbd-module
  diagrams by restriction to the sets
  \begin{align*}
    {}^{\bullet}Q=&  \{x^iy^j \mid i\leq0,\,j> 0\}\ , &
    Q^{\bullet}  =&  \{x^iy^j \mid i>0,\,j>0\} \ , \\ 
    \noalign{\smallskip}
    {}_{\bullet}Q=&  \{x^iy^j \mid i\leq0,\,j\leq 0\}\ , & 
    Q_{\bullet}  =&  \{x^iy^j \mid i> 0,\,j\leq 0\}
  \end{align*}
  which correspond to the four quadrants (with or without various
  boundary components included). We denote the restricted diagrams by
  ${\bsd}^{\bullet}$, ${}^{\bullet}{\bsd}$, ${}_{\bullet}{\bsd}$ and
  ${\bsd}_{\bullet}$ respectively. The four resulting summands are
  shown in Fig.~\ref{fig:quadrants}. The diagram appearing in the top
  right is ${\bsd}^{\bullet}$. The top left diagram is
  ${}^{\bullet}\bsd$. Similarly, ${}_{\bullet}{\bsd}$ is the bottom
  left diagram and ${\bsd}_{\bullet}$ is the bottom right diagram.
  \begin{figure}\scalebox{0.95}{%
    \begin{diagram}[PS,h=1.05em,w=0.1em,landscape]
      \by & \rTo &&          &     &          &\bvy&   \lTo   &     &&&          & \bv   &&&\;\;\quad&&&
      \ay & \rTo &&          &     &          &\avy&   \lTo   &     &&&          & \av   \\
      \dTo&      &&          &     &\ldTo(2,4)&    &          &     &&&\ldTo(6,6)& \dTo  &&&&&&
      \dTo&      &&          &     &\ldTo(2,4)&    &          &     &&&\ldTo(6,6)& \dTo  \\
      &      &&          &     &          &    &          &     &&&          &       &&&&&&
      &      &&          &     &          &    &          &     &&&          &       \\
      &      &&          &     &          &    &          &     &&&          &       &&&&&&
      &      &&          &     &          &    &          &     &&&          &       \\
      &      &&          &\bvys&          &    &          &     &&&          &       &&&&&&
      &      &&          &\avys&          &    &          &     &&&          &       \\
      &      &&\ruTo(4,2)&     &\luTo(2,2)&    &          &     &&&          &       &&&&&&
      &      &&\ruTo(4,2)&     &\luTo(2,2)&    &          &     &&&          &       \\
      \bys&      &&          &     &          &\bvs&          &     &&&          & \bvx  &&&&&&
      \ays&      &&          &     &          &\avs&          &     &&&          & \avx  \\
      \uTo&      &&          &     &\ruTo(6,6)&    &\rdTo(2,2)&     &&&\ldTo(4,2)& \uTo  &&&&&&
      \uTo&      &&          &     &\ruTo(6,6)&    &\rdTo(2,2)&     &&&\ldTo(4,2)& \uTo  \\
      &      &&          &     &          &    &          &\bvxs&&&          &       &&&&&&
      &      &&          &     &          &    &          &\avxs&&&          &       \\
      &      &&          &     &          &    &\ruTo(2,4)&     &&&          &       &&&&&&
      &      &&          &     &          &    &\ruTo(2,4)&     &&&          &       \\
      &      &&          &     &          &    &          &     &&&          &       &&&&&&
      &      &&          &     &          &    &          &     &&&          &       \\
      &      &&          &     &          &    &          &     &&&          &       &&&&&&
      &      &&          &     &          &    &          &     &&&          &       \\
      \bs & \rTo &&          &     &          &\bxs&   \lTo   &     &&&          &  \bx  &&&&&&
      \as & \rTo &&          &     &          &\axs&   \lTo   &     &&&          &  \ax  \\
      &      &&          &     &          &    &          &     &&&          &       &&&&&&
      &      &&          &     &          &    &          &     &&&          &       \\
      &      &&          &     &          &    &          &     &&&          &       &&&&&&
      &      &&          &     &          &    &          &     &&&          &       \\
      &      &&          &     &          &    &          &     &&&          &       &&&&&&
      &      &&          &     &          &    &          &     &&&          &       \\
      &      &&          &     &          &    &          &     &&&          &       &&&&&&
      &      &&          &     &          &    &          &     &&&          &       \\
      \cy & \rTo &&          &     &          &\cvy&   \lTo   &     &&&          & \cv   &&&&&&
      \dy & \rTo &&          &     &          &\dvy&   \lTo   &     &&&          & \dv   \\
      \dTo&      &&          &     &\ldTo(2,4)&    &          &     &&&\ldTo(6,6)& \dTo  &&&&&&
      \dTo&      &&          &     &\ldTo(2,4)&    &          &     &&&\ldTo(6,6)& \dTo  \\
      &      &&          &     &          &    &          &     &&&          &       &&&&&&
      &      &&          &     &          &    &          &     &&&          &       \\
      &      &&          &     &          &    &          &     &&&          &       &&&&&&
      &      &&          &     &          &    &          &     &&&          &       \\
      &      &&          &\cvys&          &    &          &     &&&          &       &&&&&&
      &      &&          &\dvys&          &    &          &     &&&          &       \\
      &      &&\ruTo(4,2)&     &\luTo(2,2)&    &          &     &&&          &       &&&&&&
      &      &&\ruTo(4,2)&     &\luTo(2,2)&    &          &     &&&          &       \\
      \cys&      &&          &     &          &\cvs&          &     &&&          & \cvx  &&&&&&
      \dys&      &&          &     &          &\dvs&          &     &&&          & \dvx  \\
      \uTo&      &&          &     &\ruTo(6,6)&    &\rdTo(2,2)&     &&&\ldTo(4,2)& \uTo  &&&&&&
      \uTo&      &&          &     &\ruTo(6,6)&    &\rdTo(2,2)&     &&&\ldTo(4,2)& \uTo  \\
      &      &&          &     &          &    &          &\cvxs&&&          &       &&&&&&
      &      &&          &     &          &    &          &\dvxs&&&          &       \\
      &      &&          &     &          &    &\ruTo(2,4)&     &&&          &       &&&&&&
      &      &&          &     &          &    &\ruTo(2,4)&     &&&          &       \\
      &      &&          &     &          &    &          &     &&&          &       &&&&&&
      &      &&          &     &          &    &          &     &&&          &       \\
      &      &&          &     &          &    &          &     &&&          &       &&&&&&
      &      &&          &     &          &    &          &     &&&          &       \\
      \cs & \rTo &&          &     &          &\cxs&   \lTo   &     &&&          &  \cx  &&&&&&
      \ds & \rTo &&          &     &          &\dxs&   \lTo   &     &&&          &  \dx  
    \end{diagram}}%
    \caption{Decomposition of $\bsd_{v_{tr}}$}
    \label{fig:quadrants}
  \end{figure}

  Note that the splitting $\bsd_{v_{tr}} = {}_\bullet \bsd \oplus {}^\bullet \bsd
  \oplus \bsd_\bullet \oplus \bsd^\bullet$ yields a corresponding
  splitting of \textsc{\v Cech} complexes
  \[\vgamma(\bsd_{v_{tr}}) = \vgamma({}_\bullet \bsd) \oplus
  \vgamma({}^\bullet \bsd) \oplus \vgamma(\bsd_\bullet) \oplus
  \vgamma(\bsd^\bullet) \ .\]
  We will show
  \begin{enumerate}[(i)]
  \item that the last three summands are acyclic so that the
    projection map $\pi \colon \vgamma(\bsd_{v_{tr}}) \rTo \vgamma({}_\bullet
    \bsd)$ is a quasi-isomorphism, and
  \item that there is a quasi-isomorphism $\beta \colon
    \vgamma({}_\bullet \bsd) \rTo A_{v_{tr}}$ such that the composite map
    \[A_{v_{tr}} \rTo^{\sigma_{v_{tr}}} \vgamma(\bsd_{v_{tr}}) \rTo^\pi \vgamma ({}_\bullet
    \bsd) \rTo^\beta A_{v_{tr}}\] is the identity.
  \end{enumerate}
  It then follows that $\sigma_{v_{tr}}$~is a quasi-isomorphism as claimed.

  \smallskip

  The main idea in both cases is to use a suitable filtration of
  diagrams, as was done implicitly in the case of an edge above. More
  precisely, we will look at a chain of epimorphisms of $R$\nbd-module
  diagrams indexed by~$\mathcal{N}_{v_{tr}}$
  \[X_0 \rTo^{\kappa_1} X_1 \rTo^{\kappa_2} \ldots \rTo^{\kappa_k}
  X_k\] such that $\vgamma (\ker \kappa_j)$ is acyclic for
  $1 \leq j \leq k$. From the short exact sequence
  \[0 \rTo \vgamma(\ker \kappa_j) \rTo \vgamma(X_{j-1})
  \rTo[l>=4em]^{\vgamma(\kappa_j)} \vgamma(X_j) \rTo 0\] we then infer that
  the map $\vgamma(\kappa_j)$~is a quasi-isomorphism.

  \smallskip

  Let us consider specifically the diagram~$X_0 = {}^\bullet \bsd$
  (top left in Fig.~\ref{fig:quadrants}). We let~$X_1$ have the same
  entries and structure maps as~$X_0$ except at the flags $\{v_{tr},\,
  e_r\}$ and~$\{v_{tr},\, e_r,\, S\}$ where $X_1$ is trivial (see
  Fig.~\ref{fig:nervevert} for a reminder on the indexing); the
  incident structure maps are forced to be zero maps, of course. The
  map $\kappa_1$ is the identity where possible, or else the zero
  map. The \textsc{\v Cech} complex $\vgamma (\ker \kappa_1)$ is a
  two-step complex with an isomorphism as differential and is thus
  acyclic. --- We construct further diagrams $X_j$ in a similar manner
  from~$X_{j-1}$, by prescribing two flags $\tau_1$ and~$\tau_2$ on
  which the former differs from the latter in taking the zero module
  as value, and by declaring $\kappa_j$ to be the identity where
  possible. In detail, we choose\goodbreak
  \begin{itemize}
  \item[$j=2$:] $\tau_1 = \{ e_r\}$ and $\tau_2 = \{e_r,\, S \}$;
  \item[$j=3$:] $\tau_1 = \{ v_{tr},\, S\}$ and $\tau_2 = \{ v_{tr},\, e_t,\, S\}$;
  \item[$j=4$:] $\tau_1 = \{ S\}$ and $\tau_2 = \{ e_t,\, S\}$.
  \end{itemize}
  This makes $X_4$ the trivial all-zero diagram so that
  $\vgamma({}^\bullet \bsd)$ is quasi-isomorphic to the zero complex
  via $\vgamma(\kappa_4 \kappa_3 \kappa_2 \kappa_1)$. --- The diagrams
  $\bsd^\bullet$ and~$\bsd_\bullet$ can be dealt with in a similar
  manner. This proves~(i).

  \smallskip

  To prove~(ii) we employ a suitable filtration of $X_0 = {}_\bullet
  \bsd$: we let $X_j$ and $\kappa_j$ be determined in the manner
  described above by the choices
  \begin{itemize}
  \item[$j=1$:] $\tau_1 = \{ v_{tr},\, e_r\}$ and $\tau_2 = \{ v_{tr},\, e_r,\, S\}$;
  \item[$j=2$:] $\tau_1 = \{ e_r\}$ and $\tau_2 = \{ e_r,\, S\}$;
  \item[$j=3$:] $\tau_1 = \{ v_{tr},\, e_t\}$ and $\tau_2 = \{ v_{tr},\, e_t,\, S\}$;
  \item[$j=4$:] $\tau_1 = \{ e_t\}$ and $\tau_2 = \{ e_t,\, S\}$;
  \item[$j=5$:] $\tau_1 = \{ v_{tr}\}$ and $\tau_2 = \{ v_{tr},\, S\}$.
  \end{itemize}
  The diagram~$X_5$ has a single non-trivial entry, {\it viz.}, the
  entry~$A_{v_{tr}}$ at position~$S$ so that $\vgamma (X_5) =
  A_{v_{tr}}$. The map $\beta = \vgamma (\kappa_5 \kappa_4 \kappa_3
  \kappa_2 \kappa_1)$ satisfies all the required properties.
\end{proof}

We also need to record naturality properties of the
maps~$\sigma_F$. Let $F \subseteq G$ be faces of~$S$. Since every
$\tau \in \mathcal{N}_G$ is also an element of~$\mathcal{N}_F$, we can
define a ``projection'' map
\begin{equation}
  \label{eq:proj_NF_NG}
  \lambda_{FG} \colon \vgamma (E_F) \rTo \vgamma (E_G)
\end{equation}
which maps summands occurring in both complexes by the identity, and
maps all other summands of the source to~$0$. (Note that
$\vgamma(E_G)$ is, in general, not a direct summand of~$\vgamma
(E_F)$.) It is a matter of straightforward checking that
$\sigma_G|_{A_F} = \lambda_{FG} \circ \sigma_F$, and that we have
\begin{equation}
  \label{eq:lambda_functorial}
  \lambda_{FH} = \lambda_{FG} \circ \lambda_{GH}
\end{equation}
for a triple of faces $F \subseteq G \subseteq H$ of~$S$.

\section{Partial totalisations of triple complexes. Applications}
\label{sec:final_proof}

For this last section of the paper we assume throughout that $C$
and~$Y$ are as in Construction~\ref{con:extend}; that is, we assume
that $C$ is a bounded cochain complex of finitely generated free
$\LL$\nbd-modules, with $C^n$ having basis~$B_n$, and that $Y$~is a
bounded $\fp$\nbd-diagram in the sense of
Definition~\ref{def:presheaf} with $Y_S = C$, with each $Y^n$
isomorphic to a $B_n$\nbd-indexed direct sum of diagrams of the
form~$D(k_n)$. We write $s_{FG} \colon Y_F \rTo Y_G$ for the structure
map associated to the inclusion of faces $F \subseteq G$.

\medbreak

We will introduce the following complexes and maps between them:
\begin{gather*}
  \bfree \rTo^\simeq \totds S^{\ast,\ast,\ast} \rTo^\simeq \totds
  T^{\ast,\ast,\ast} \rTo^\simeq \totds U^{\ast,\ast,\ast} \\
  \noalign{\noindent and} \totds V^{\ast,\ast,\ast} \iso \totds
  W^{\ast,\ast} \lTo^\simeq C \ ,\\
  \noalign{\noindent and a splitting of complexes} \totds
  U^{\ast,\ast,\ast} \iso \totds V^{\ast,\ast,\ast} \oplus \,?\,\ .
\end{gather*}
Provided all the cochain complexes listed in~\eqref{eq:cond_edge}
and~\eqref{eq:cond_vertex} are acyclic, the maps marked ``$\simeq$''
above are quasi-isomorphisms; hence we obtain, in the derived category
of~$R$, morphisms $C \rTo^{\bar s} B' \rTo^{\bar r} C$ with $\bar r
\circ \bar s = \id_C$. Since both $C$ and~$B'$ are bounded complexes
of free $R$\nbd-modules we can lift these morphisms to $R$\nbd-linear
maps $C \rTo^s B' \rTo^r C$ with $r \circ s \simeq \id_C$. Since
$B'$~is finitely generated this shows $C$~to be $R$\nbd-finitely
dominated \cite[Proposition~3.2]{MR815431}, thereby finishing the
proof of the implication (b)~$\Rightarrow$~(a) of the Main Theorem.

\medbreak

We start with the triple complex
\[S^{u,s,t} =
\begin{cases}
  0 & \text{for } t \neq 0 \ , \\
  \bigoplus_{\substack{F \subseteq S \\ \dim F = u}} Y^s_F & \text{for
  } t = 0 \ .
\end{cases}\] Differentials are necessarily trivial in
$z$\nbd-direction; for fixed~$s$, the complex $S^{\ast,s,0}$ is the
\textsc{\v Cech} complex of the constant $\fp$\nbd-indexed
diagram~$Y^s$, and for fixed~$u$ the complex $S^{u,\ast,0}$ is a
direct sum of complexes~$Y_F$ with differential changed by the
sign~$(-1)^u$.

We note that $S^{\ast,\ast,\ast}$ is actually a double complex in
disguise, and that $\totds S^{\ast,\ast,\ast} = \totds S^{\ast,\ast,0}
= \vgamma (Y)$. Hence by Proposition~\ref{prop:bdfgfree}:

\begin{lemma}
  There exists a bounded cochain complex~$\bfree$ of finitely
  generated projective $R$\nbd-modules, together with a
  quasi-isomorphism
  \[\chi \colon \bfree \rTo \totds S^{\ast,\ast,\ast} \
  .\tag*{\qedsymbol}\]
\end{lemma}

\medbreak

If the face~$F$ is contained in the minimal element of the flag~$\tau$
then $A_F \subseteq \comp{\tau}$. Consequently, we can define a triple
complex $T^{\ast,\ast,\ast}$ by
\begin{align}
  T^{u,s,t} &=\bigoplus_{\substack{F\subset S \\ \dim F=u}} \Big( Y_F^s
  \otimes_{A_F} \bigoplus_{\substack{\tau \in \mathcal{N}_F \\ \dim \tau=t}}
  A\langle\tau\rangle \Big)\\
  \noalign{\noindent with differentials induced by}
  d_x &= [F:G] \cdot (s_{FG} \otimes \lambda_{FG})   \ ,     \notag\\  
  d_y &= (-1)^{u} \cdot (d_F  \otimes 1)    \ , \text{ and} \notag\\  
  d_z &= (-1)^{u+s} \cdot (1 \otimes d_{\mathcal{N}_F}) \ ,    \notag    
\end{align}
with $\lambda_{FG}$~is as in~\eqref{eq:proj_NF_NG}, where
$d_{\mathcal{N}_F}$ denotes the differential of the cochain
complex~$\vgamma (\bsd_F)$ and $d_F$ denotes the differential of the
cochain complex~$Y_F$. (Note that for fixed~$t$ the maps~$d_x$ are
induced by the cochain complex maps
\[[F:G] \cdot \big( s_{FG} \tensor \lambda_{FG} \big) \colon Y_F
\tensor_{A_F} \comp{\tau} \rTo Y_G \tensor_{A_G} \comp{\tau} \ ,\]
where $\tau$~is a $t$\nbd-dimensional flag in~$\mathcal{N}_G$; this
implies, in view of~(\ref{eq:lambda_functorial}), that $d_x \circ d_x
= 0$.)

\medbreak

The triple complex just defined is such that $T^{u,s,*}$, for fixed
indices~$u$ and~$s$, is a direct sum of complexes of the form $Y^s_F
\tensor_{A_F} \vgamma (E_F)$, with differential changed by a sign
$(-1)^{u+s}$. But as $\vgamma (E_F)$~is quasi-isomorphic to~$A_F$ via
the map~$\sigma_F$ defined in~\eqref{eq:sigma_F}, and as $Y^s_F$~is a
free $A_F$\nbd-module we have a quasi-isomorphism $Y^s_F \tensor_{A_F}
\vgamma (E_F) \simeq Y^s_F$. In fact, the compositions
\[Y_F^s \iso Y_F^s \tensor_{A_F} A_F \rTo^{\id \tensor \sigma_F} Y_F^s
\tensor_{A_F} \vgamma (E_F)\] assemble to a map of triple complexes
\[\upsilon \colon S^{\ast,\ast,\ast} \rTo T^{\ast,\ast,\ast}\] which
is a quasi-isomorphism on complexes in $z$\nbd-direction in the sense
of Lemma~\ref{lem:triple_map}. We thus have:

\begin{lemma}
  The map $\totds (\upsilon) \colon \totds S^{\ast,\ast,\ast} \rTo
  \totds T^{\ast,\ast,\ast}$ is a quasi-isomorphism.\qed
\end{lemma}

\medbreak

Next, we define a triple complex~$U^{\ast,\ast,\ast}$ which is,
informally speaking, the restriction of~$T^{\ast,\ast,\ast}$ to those
flags which are either zero-dimensional, or do not involve~$S$.
Explicitly,
\begin{equation}
  \label{eq:triple_U}
  U^{u,s,t} =
  \begin{cases}
    0 & \text{for } t \neq 0,1 \ ,\\
    \bigoplus_{\substack{F\subset S \\ \dim F=u}} \Big( Y_F^s
    \otimes_{A_F} \bigoplus_{\substack{\tau \in \mathcal{N}_F \\ \dim
        \tau = 1 \\ S \notin \tau}} \comp{\tau} \Big)
    & \text{for } t = 1 \ , \\
    \bigoplus_{\substack{F\subset S \\ \dim F=u}} \Big( Y_F^s
    \otimes_{A_F} \bigoplus_{G \supseteq F} \comp{G} \Big)
    & \text{for } t = 0 \ ;
  \end{cases}
\end{equation}
note that $U^{u,s,1} = 0$ if $u \neq 0$, and that the second direct
sum in the last line is taken over all $0$\nbd-dimensional flags
in~$\mathcal{N}_F$, \ie, over all faces~$G$ containing~$F$. The
differentials are either trivial by necessity, or the restrictions of
the corresponding differentials of~$T^{\ast,\ast,\ast}$ where
possible.

There is an obvious ``projection'' map of triple complexes
\[\omega \colon T^{\ast,\ast,\ast} \rTo U^{\ast,\ast,\ast} \ ;\] it is
given by sending the summand $Y^s_F \tensor_{A_F} \comp{\tau}$ to
itself via the identity map if the target contains the same summand,
and by sending it to~$0$ otherwise. (Note that $U^{\ast,\ast,\ast}$~is
not a direct summand of~$T^{\ast,\ast,\ast}$ due to the presence of
too many non-trivial differentials in $z$\nbd-direction in the
latter.)

\begin{lemma}
  If all the cochain complexes listed in~\eqref{eq:cond_edge}
  and~\eqref{eq:cond_vertex} are acyclic, the map $\totds (\omega)$ is
  a quasi-isomorphism.
\end{lemma}

\begin{proof}
  First note that the complexes listed in~\eqref{eq:cond_edge} are of
  the type $C \tensor \comp{e,S}$ for an edge~$e$ of~$S$ (all unmarked
  tensor products are over the ring $A_S = R[x,\,x\inv,\,y,\,y\inv]$
  in this proof). Similarly, the complexes listed
  in~\eqref{eq:cond_vertex} are of the form $C \tensor \comp{v,S}$ for
  $v$~a vertex of~$S$.

  Suppose now that these complexes are acyclic. Since $C$~is a bounded
  complex of free modules, they are then in fact contractible, \ie,
  homotopy equivalent to the trivial complex. Since tensor products
  preserve homotopies, it follows that for each $2$\nbd-dimensional
  flag $\tau = (v \subset e \subset S)$ the cochain complex
  \[ Y_S \otimes_{A_S} \comp{\tau} \iso Y_S \otimes_{A_S} \comp{v,S}
  \otimes_{\comp{v,S}} \comp{\tau}\] is contractible and hence
  acyclic. That is, acyclicity of the complexes \eqref{eq:cond_edge}
  and~\eqref{eq:cond_vertex} implies acyclicity of the additional
  eight complexes
  \begin{align*}
    & C \tensor R\nov{x}\nov{y} \ , && C \tensor R\nov{x}\nov{y\inv}  \ , \\
    & C \tensor R\nov{x\inv}\nov{y} \ ,  && C \tensor R\nov{x\inv}\nov{y\inv} \ , \\
    & C \tensor R\nov{y}\nov{x} \ ,  && C \tensor R\nov{y}\nov{x\inv} \ , \\
    & C \tensor R\nov{y\inv}\nov{x} \ ,  && C \tensor
    R\nov{y\inv}\nov{x\inv} \ .
  \end{align*}

  Let $F$ be a face of~$S$. If $\tau \in \mathcal{N}_F$ denotes a
  positive-dimensional flag ending in~$S$, we know that $A_S \subset
  \comp{\tau}$ and thus
  \begin{equation*}
    Y_F \tensor_{A_F} \comp{\tau} \iso Y_F \tensor_{A_F} A_S
    \tensor_{A_S} \comp{\tau}  \iso C \tensor_{A_S} \comp{\tau} \simeq
    0
  \end{equation*}
  where we made use of the fact that $Y_F \tensor_{A_F} A_S \iso Y_S =
  C$ according to Construction~\ref{con:extend}.

  But this means that $\omega$~is a quasi-isomorphism of cochain
  complexes
  \[T^{u,\ast,t} \rTo U^{u,\ast,t}\] for all $u,t \in \bZ$. Indeed,
  for $t=0$ it is an identity map, for $t=1$ it is a direct sum of
  identity maps (corresponding to summands indexed by flags of the
  form $v \subset e$) and maps from acyclic to trivial complexes (all
  other summands), for $t=2$ it is a map from an acyclic to a trivial
  one, using the results of the previous two paragraphs. From
  Lemma~\ref{lem:triple_map} we conclude that $\totds (\omega)$ is a
  quasi-isomorphism as claimed.
\end{proof}

Now $U^{\ast,\ast,\ast}$ has a direct summand consisting, informally
speaking, of the summands indexed by $G=S$ at height $t=0$ only:
\begin{equation}
  \label{eq:triple_V}
  V^{u,s,t} =
  \begin{cases}
    0 & \text{for } t \neq 0 \ ,\\
    \bigoplus_{\substack{F\subset S \\ \dim F=u}} \big( Y_F^s
    \otimes_{A_F} \comp{S} \big) & \text{for } t = 0 \ ;
  \end{cases}
\end{equation}
differentials are obtained by restricting the corresponding ones
of~$U^{\ast,\ast,\ast}$. Clearly then $\totds V^{\ast,\ast,\ast}$ is a
direct summand of~$\totds U^{\ast,\ast,\ast}$.

The triple-complex totalisation~$V^{\ast,\ast,\ast}$ agrees with the
double complex totalisation $\totds V^{\ast,\ast,0}$ (due to the
absence of non-trivial terms for $z \neq 0$). Also, we have $\comp{S}
= A_S$ and $Y_F^s \tensor_{A_F} A_S \iso C^s$, by construction of~$Y$;
it follows that the double complex~$V^{\ast,\ast,0}$ is isomorphic to
the double complex $W^{\ast,\ast}$ which, in column~$p$, has a direct
sum of copies of~$C$ indexed by the $p$\nbd-dimensional faces of~$S$,
with differential changed by the sign $(-1)^p$, and has in row~$q$ the
\textsc{\v Cech} complex of the constant $\fp$\nbd-indexed diagram
with value~$C^s$.

\begin{lemma}
  There is a quasi-isomorphism $C \rTo \totds (W^{\ast,\ast})$.
\end{lemma}

\begin{proof}
  Since $C$ is bounded we may, by simple re-indexing, assume that $C$
  is concentrated in non-negative degrees and that consequently
  $W^{\ast,\ast}$ is concentrated in the first quadrant. For $q \geq
  0$ let $h_q$ denote the diagonal inclusion $C^q \rTo W^{0,q} =
  \oplus_v C^q$, modified by a sign~$(-1)^q$ (the direct sum taken
  over all vertices of~$S$). By construction the two composites
  \[C^q \rTo W^{0,q} \rTo W^{0,q+1} \quad \text{and} \quad C^q \rTo
  C^{q+1} \rTo W^{0,q+1}\] agree up to sign. The
  complexes
  \[0 \rTo C^q \rTo W^{0,q} \rTo W^{1,q} \rTo \cdots\] are exact; this
  follows, for example, from the observation that they can be obtained
  by tensoring the dual of the augmented cellular chain complex of~$S$
  (which computes $\tilde H_*(S; R) = 0$) with the free
  $R$-module~$C^q$. --- We can now apply
  Proposition~\ref{lem:augmented_double} to conclude that $C \simeq
  \totds W^{\ast,\ast}$ as claimed.
\end{proof}

\providecommand{\bysame}{\leavevmode\hbox to3em{\hrulefill}\thinspace}

\end{document}